\documentclass[11pt]
{article}

\newsavebox{\savepar}

\usepackage{amsmath}
\usepackage{amsfonts}
\usepackage{euscript}
\usepackage{oldgerm}
\usepackage{eufrak}
\usepackage{amsthm}
\usepackage{rotating}
\usepackage{calrsfs}
\DeclareMathAlphabet{\pazocal}{OMS}{zplm}{m}{n}

\makeindex

  \newtheorem{theorem}{Theorem}[section]

  \theoremstyle{definition}

  \theoremstyle{remark}
  \newtheorem{remark}[theorem]{Remark}

\newtheorem{lemma}[theorem]{Lemma}
\newtheorem{proposition}[theorem]{Proposition}

\theoremstyle{definition}

\theoremstyle{remark}

\begin{document}

\newcommand{\namelistlabel}[1]{\mbox{#1}\hfil}
\newenvironment{namelist}[1]{%
\begin{list}{}
{
\let\makelabel\namelistlabel
\settowidth{\labelwidth}{#1}
\setlength{\leftmargin}{1.1\labelwidth}
}
}{%
\end{list}}

\newcommand{\inp}[2]{\langle {#1} ,\,{#2} \rangle}
\newcommand{\vspan}[1]{{{\rm\,span}\{ #1 \}}}

\newcommand{\R} {{\mathbb{R}}}

\newcommand{\B} {{\mathbb{B}}}
\newcommand{\C} {{\mathbb{C}}}
\newcommand{\N} {{\mathbb{N}}}
\newcommand{\Q} {{\mathbb{Q}}}
\newcommand{\Z} {{\mathbb{Z}}}

\newcommand{\BB} {{\mathcal{B}}}

\title{ Discrete Modified Projection Method for   Urysohn Integral 
Equations with  Smooth Kernels }
\author{  Rekha P.  KULKARNI \thanks{Department of Mathematics, I.I.T. Bombay, 
Powai, Mumbai 400076, India,  rpk@math.iitb.ac.in,   } and Gobinda RAKSHIT \thanks{
gobindarakshit@math.iitb.ac.in} 
}

\date {}
\maketitle

\begin{abstract}

Approximate solutions of linear and nonlinear integral equations using methods related to 
an interpolatory projection involve many integrals which need to be evaluated using a numerical quadrature formula. 
In this paper, we consider discrete versions of the modified projection method and of the iterated modified projection method
for solution of a Urysohn integral equation with a smooth kernel. For $r \geq 1,$ a space of piecewise polynomials of 
degree $\leq r - 1$ with respect to an uniform partition is chosen to be the approximating space and the projection 
is chosen to be the interpolatory projection at $r$ Gauss points. 
The orders of convergence which we 
obtain for these discrete versions indicate the choice of numerical quadrature which preserves 
the orders of convergence. 
 Numerical results are given for a specific example.
\end{abstract}



\newpage
\setcounter{equation}{0}
\section {Introduction}

Let $X = L^\infty [a, b]$, and 
consider a Urysohn integral operator 
\begin{equation}\label{eq:1.1}
\mathcal{K} (x)(s)  = \int_a^b \kappa (s, t, x (t))  d t, \;\;\; s \in [a, b], \; x \in X,
\end{equation}
where $\kappa (s, t, u)$ is a continuous real valued function
defined on 
$$ \Omega = [a, b] \times [a, b] \times \R.$$
Then  $\mathcal{K} $ is a compact operator from $L^\infty [a, b]$ to $C [a, b].$ 
Assume that for $ f \in C[a, b],$
\begin{equation}\label{eq:1.2}
x - \mathcal{K} (x) = f
\end{equation}
has a unique solution $\varphi.$

We are interested in approximate solutions of the above equation. We consider projection methods associated 
with a sequence of interpolatory projections converging to the Identity operator pointwise.

For $r \geq 1,$ let $X_n$ denote the space of piecewise polynomials of degree $\leq r - 1$ with respect to 
a uniform partition of $[a, b]$ with $n $ subintervals. Let $\displaystyle {h = \frac {b - a}{n}}$ denote the 
length of each subinterval of the above partition. Let $Q_n: C [a, b] \rightarrow X_n$ denote
the interpolation operator  at $r$ Gauss points.  Then in the collocation method, (\ref{eq:1.2}) is approximated by
\begin{equation}\label{eq:1.3}
\varphi_n^C - Q_n \mathcal{K} (\phi_n^C) = Q_n f.
\end{equation}
The iterated collocation solution is defined as
\begin{equation}\label{eq:1.4}
\varphi_n^S =  \mathcal{K} (\phi_n^C) + f.
\end{equation}
In Grammont et al \cite{Gram3}  the following modified projection method is investigated:
\begin{equation}\label{eq:1.5}
\varphi_n^M - \mathcal{K}_n^M (\phi_n^M) = f,
\end{equation}
where
\begin{equation}\label{eq:1.6}
\mathcal{K}_n^M (x) = Q_n \mathcal{K} (x) + \mathcal{K} (Q_n x) - Q_n \mathcal{K} (Q_n x).
\end{equation}
It is a generalization of the modified projection method in the linear case, which was proposed in Kulkarni \cite {Kul1}.
The iterated monified projection solution is defined as
\begin{equation}\label{eq:1.7}
\tilde {\varphi}_n^M = \mathcal{K} (\phi_n^M) + f.
\end{equation}
If $\displaystyle { \frac {\partial \kappa} {\partial u} \in C^r (\Omega)}$ and $ f \in C^r [a, b],$ then
the following orders of convergence can be obtained from  Atkinson-Potra \cite{AtkP1}:
\begin{equation}\label{eq:1.8}
  \| \varphi - \varphi_n^C \|_\infty= O (h^r), \;\;\; \| \varphi - \varphi_n^S \|_\infty= O (h^{2 r}).
\end{equation}
The following estimates  are obtained in Grammont et al \cite{Gram3}:
Let $ f \in C^{ 2 r} [a, b].$

If $\displaystyle { \frac {\partial \kappa} {\partial u} \in C^{2 r} (\Omega)}$ and  then
\begin{equation}\label{eq:1.9}
\| \varphi - \varphi_n^M \|_\infty= O (h^{3r}), 
\end{equation}
whereas if $\displaystyle {\frac {\partial \kappa} {\partial u} \in C^{3 r} (\Omega),}$
then
\begin{equation}\label{eq:1.10}
 \| \varphi - \tilde{\varphi}_n^M \|_\infty= O (h^{4 r}),
\end{equation}
In practice, it is necessary to replace the integral in the definition of $ \mathcal{K}$ by a numerical quadrature formula, giving rise to 
a discrete version of the above methods. The discrete version of the iterated collocation method is considered in Atkinson-Flores \cite {AtkF}.  Our aim is to  investigate a choice of a numerical quadrature rule which preserves 
the above orders of convergence in the discrete versions of the modified projection and the iterated modified projection methods. 

We choose a composite numerical quadrature formula with a degree of precision $d$ and with respect to a uniform partition of
$[a, b]$ with $m$ subintervals. Let $\tilde{h}$ denote the length of each subinterval of this partition. The discrete modified projection solution and the iterated discrete modified projection solution are denoted repectively by $z_n^M$ and $\tilde{z}_n^M.$ 
We prove the following orders of convergence.
Let $\displaystyle {\kappa \in C^d (\Omega),\; \mbox { and } \;  f \in C^d [a, b].}$ 

If 
$ \displaystyle {\frac {\partial^2 \kappa } {\partial u^2} \in C^{2 r } (\Omega ),}$ then 
\begin{equation}\label{eq:1.11}
\| \varphi - z_n^M \|_\infty= O \left ( \max \left \{ \tilde{h}^d, h^{ 3 r} \right \} \right ),
\end{equation}
whereas if
$ \displaystyle {\frac {\partial^2 \kappa } {\partial u^2} \in C^{3 r } (\Omega ),}$ then 
\begin{equation}\label{eq:1.11}
\| \varphi - \tilde{z}_n^M \|_\infty
= O \left ( h^r \max \left \{ \tilde{h}^d, h^{ 3 r} \right \} \right ).
\end{equation}
Thus, the orders of convergence in (\ref{eq:1.9}) and (\ref{eq:1.10}) are preserved provided the numerical quadrature rule is so chosen that
$\tilde{h}^d = h^{ 3 r}.$

In Chen et al \cite {Chen} discrete versions of the modified projection and the iterated modified projection methods for solutions of linear integral equations  are
considered. They consider a slightly restrictive case where the same uniform partition of $[a, b]$ is
considered to define a composite numerical quadrature formula and an interpolatory projection, that is $\tilde {h} = h.$  We allow these partitions to be different. More precisely, we choose $ m = n p, \; p \in \N,$ instead of $ m = n.$

 The emphasis of this paper is on the solution of Urysohn integral equations and we include 
the results about the linear case for the sake of completeness and for their use to treat the nonlinear case.

The paper has been arranged in the following way. In Section 2.1, we describe the Nystr\"{o}m approximation of $ \mathcal{K}$ obtained by replacing the
integral by a numerical quadrature. In Section 2.2, this approximation is used to define 
 the discrete versions of projection methods  (\ref{eq:1.3})-(\ref{eq:1.7}).
Section 3 is devoted to approximate solutions of linear integral equations using the discrete versions of  the modified projection and the iterated modified projection methods. The order of convergence for the discrete modified projection solution of a Urysohn integral equation 
is obtained in Section 4, whereas the discrete iterated modified projection method is investigated in Section 5. Numerical results are given in Section 6.


\setcounter{equation}{0}
\section{Methods of Approximation}

\subsection {Nystr\"{o}m Approximation}
 
Assume that the kernel $\kappa$ of $ \mathcal{K}$  defined by (\ref{eq:1.1}) is such that 
\begin{equation}\nonumber
\frac {\partial^2 \kappa} {\partial u^2} \in C (\Omega).
\end{equation}
It is also assumed that $1$ is not an eigenvalue of the compact linear operator $\mathcal{K}' (\varphi)$ so that 
  $(I - \mathcal{K}' (\varphi))^{-1}: C[a, b] \rightarrow C[a, b]$ is a bounded linear operator.

Let $m \in \mathbb{N},$ and consider the following uniform partition of $[a, b]:$
\begin{equation}\label{eq:2.2}
a = s_0 < s_1 < \cdots < s_m = b.
\end{equation}
Let 
\begin{equation}\nonumber
 \tilde h = s_{k} - s_{k-1} = \frac {b - a} {m}, \; k = 1, \ldots, m.
\end{equation}
Consider a basic quadrature rule
  \begin{equation}\label{eq:2.4}
  \int_0^1 f (t) d t \approx \sum_{i=1}^\rho w_i f (\mu_i)
  \end{equation}
  which has a degree of precision $ 2 r - 1$ or higher.
  
Let
\begin{equation}\label{eq:2.5}
\zeta_i^j  = s_{j-1} +   \mu_i \;  \tilde h,  \;\;\; i = 1, \ldots, \rho, \;\;\; j = 1, \ldots, m.
\end{equation}
A composite integration rule with respect to the partition  (\ref{eq:2.2}) 
is then defined as
\begin{eqnarray}\label{eq:2.6}
 \int_a^b f (t) d t &=& \sum_{j=1}^m \int_{s_{j-1}}^{s_j} f (t) d t 
  \approx  \tilde h \sum_{j=1}^m  \sum_{i=1}^\rho w_i \; f (\zeta_i^j ).
\end{eqnarray}
The error in the numerical quadrature  is assumed to be of the following form: 

There is an integer  $ d \geq 2r$ such that    if $f \in C^{d} [a, b],$
then
\begin{equation}\label{eq:2.7}
\left | \int_a^b f (t) d t 
 - \tilde h \sum_{j=1}^m  \sum_{i=1}^\rho w_i \; f (\zeta_i^j ) \right | \leq C_1 \|f^{(d)} \|_\infty\tilde {h}^{d},
\end{equation}
where  $f^{(d)} $ denotes $d$ th derivative of $f$ and $C_1$ is a constant independent of $\tilde{h}.$

If  the basic quadrature rule in (\ref{eq:2.4})  is chosen to be the Gaussian quadrature with $\rho = r$
or the Newton-Cotes quadrature with  $\rho = 2 r,$ then (\ref{eq:2.7}) is satisfied with $d = 2 r.$

 We replace the integral in (\ref{eq:1.1}) by the numerical quadrature formula (\ref{eq:2.6}) and define
the Nystr\"{o}m operator as
\begin{equation}\label{eq:2.8}
\mathcal{K}_m (x) (s)  =  \tilde {h}  \sum_{j=1}^m  \sum_{i=1}^\rho w_i \;  \kappa \left (s,  \zeta_i^j, x \left (\zeta_i^j \right ) \right).
\end{equation}
Then $\mathcal{K}_m$ has the following properties.
\begin{enumerate}
\item $\mathcal{K}_m: C[a, b] \rightarrow C[a, b]$ are completely continuous operators.

\item $\mathcal{K}_m$ is a colloctively compact family: 
$\mathcal{B} \subset C([a, b])$ bounded implies that the set
$\{ \mathcal{K}_m (\mathcal{B}): m \geq 1 \}$ is relatively compact in $C [a, b].$

\item At each $x \in C [a, b],$ $\mathcal{K}_m (x) \rightarrow \mathcal{K} (x)$ as $ m \rightarrow \infty.$

\item   $\{ \mathcal{K}_m  \} $ is an equicontinuous family at each $ x \in C[a, b].$

\item  For $m \geq 1,$ $ \mathcal{K}_m $
are twice Fr\'echet  differentiable: 
\begin{equation}\label{eq:2.9}
\mathcal {K}_m'  (x) v (s)  =  \tilde {h}  \sum_{j=1}^m  \sum_{i=1}^\rho  w_i \;  \frac {\partial \kappa } {\partial u} (s,  \zeta_i^j, x (\zeta_i^j)) v (\zeta_i^j), \;\;\; s \in [a, b].
\end{equation}
and
\begin{equation}\label{eq:2.10}
\mathcal {K}_m''  (x) (v_1, v_2) (s)  =  \tilde {h}  \sum_{j=1}^m  \sum_{i=1}^\rho  w_i \;  \frac {\partial^2 \kappa } {\partial u^2} (s,  \zeta_i^j, x (\zeta_i^j)) v_1 (\zeta_i^j) \; v_2 (\zeta_i^j), \;\;\; s \in [a, b].
\end{equation}

\end{enumerate}
Let
$\kappa \in C^{d} (\Omega) \; \mbox {and} \; f \in C^{d} [a, b].$ 
Then $\mathcal{K} $ is a compact operator from $L^\infty [a, b]$ to $C^{d} [a, b]$ 
and the exact solution $\varphi$ of (\ref{eq:1.2}) belongs to $C^{d} [a, b].$ 
From (\ref{eq:2.7}) we conclude that
\begin{eqnarray}\label{eq:2.11}
\|\mathcal{K} (\varphi)  - \mathcal{K}_m (\varphi ) \|_\infty  = O \left (\tilde{h}^{d} \right).
\end{eqnarray}
In the Nystr\"{o}m method,  (\ref{eq:1.2}) is approximated by
\begin{equation}\label{eq:2.12}
x_m - \mathcal{K}_m (x_m) = f.
\end{equation}
We quote the following result from Atkinson \cite {Atk1}.

Fix $\delta > 0$ and define
 \begin{equation}\label{eq:2.13}
B (\varphi, \delta ) = \{ x: \|x - \varphi \|_\infty \leq \delta \}.
\end{equation} 

There exists a positive integer $m_0$ such that for $ m \geq m_0,$ the  equation (\ref{eq:2.12})  has a unique solution $\varphi_m$ in $B (\varphi, \delta )$
and
\begin{eqnarray}\label{eq:2.14}
\|\varphi - \varphi_m \|_\infty  = O \left (\tilde{h}^{d} \right).
\end{eqnarray}

\newpage

\subsection{Discrete Projection Methods}

We first define  the interpolatory projection at $r$ Gauss points.

Let $ n \in \mathbb{N}$ and consider the following uniform partition of $[a, b]:$
\begin{equation}\label{eq:2.15}
a = t_0 < t_1 < \cdots < t_n = b.
\end{equation}
Let
\begin{equation}\nonumber
h =  t_{k} - t_{k-1} =  \frac {b - a} {n}, \; \;\; k = 1, \ldots, n.
\end{equation}
 Let $ r $ be a positive integer 
and let 
 \begin{equation}\nonumber
  \mathcal{X}_n = \left \{  g \in L^\infty [a, b]: g |_{[t_{k-1}, t_k]} \; \mbox{is a polynomial of degree} \; \leq r - 1, \; k = 1, \ldots, n
  \right \}.
  \end{equation}
  Let
   \begin{equation}\label{eq:2.17}
   0 < q_1 < \cdots < q_r < 1
   \end{equation}
  denote the Gauss-Legendre zeros of order $r$ in $[0, 1].$ 
  The  collocation nodes are chosen as follows:
   \begin{equation}\label{eq:2.18}
  \tau_i^k  = t_{k-1} + q_i \; h, \;\;\; i = 1, \ldots, r, \;\;\; k = 1, \ldots, n.
  \end{equation}
  Define the interpolation operator $Q_n: C[a, b] \rightarrow \mathcal{X}_n$ as
   \begin{equation}\label{eq:2.19}
   (Q_n x) (\tau_i^k) = x (\tau_i^k), \;\;\; i = 1, \ldots, r, \;\;\; k = 1, \ldots, n.
   \end{equation}
Using the Hahn-Banach extension theorem, as in Atkinson et al \cite {AtkG}, $Q_n$ can be extended to $L^\infty [a, b].$  
Note that  for $x \in C [a, b],$ $Q_n x \rightarrow x$ as $ n \rightarrow \infty.$ 
As a consequence, 
\begin{equation}\label{eq:2.20}
\sup_{n} \|Q_n |_{C[a, b]}\| \leq q.
\end{equation}
Let 
$$ p \in \N \; \mbox {and} \; m = p n.$$ 
Replacing the operator $\mathcal{K}$ by the Nystr\"{o}m operator $\mathcal{K}_m$ from (\ref{eq:2.8}) in (\ref{eq:1.3})-(\ref{eq:1.7}) we obtain 
discrete versions of various projection methods as given below.

\noindent
    {Discrete Collocation method:}
\begin{equation}\label{eq:2.21}
z_n^C - Q_n \mathcal{K}_m (z_n^C) = Q_n f.
\end{equation}
Discrete Iterated  Collocation method:
\begin{equation}\label{eq:2.22}
 {z}_n^S =  \mathcal{K}_m ( {z}_n^C) +  f.
\end{equation}
 The discrete modified projection operator is defined as 
\begin{equation}\label{eq:2.23}
\tilde{\mathcal{K}}_n^M (x) = Q_n {\mathcal{K}}_m (x) + {\mathcal{K}}_m (Q_n x) - Q_n  {\mathcal{K}}_m (Q_n x).
\end{equation}
Discrete Modified Projection Method: We later show that for $n$ and $m$ big enough, 
\begin{equation}\label{eq:2.24}
 x_n -  \tilde{\mathcal{K}}_n^M (x_n) =  f
\end{equation}
has a unique solution ${z}_n^M$ in a neighbourhood of the exact solution $\varphi.$

Discrete Iterated  Modified Projection Method:
\begin{equation}\label{eq:2.25}
 \tilde{z}_n^M =  {\mathcal{K}}_m ( {z}_n^M) +  f.
\end{equation}
Our aim is to show that $z_n^M \rightarrow \varphi,$ $\tilde{z}_n^M \rightarrow \varphi$ and obtain their orders of convergence. 
We first consider the case of a linear  integral operator. The results proved for the linear integral equation are needed in the case of the Urysohn integral equation.


\setcounter{equation}{0}
\section{Linear Integral Equations}

Let   $\kappa (\cdot, \cdot ) \in C ([a, b] \times  [a, b]).$
Then the integral operator
\begin{equation}\nonumber
(\mathcal{K}x)(s)  = \int_a^b \kappa (s, t) x (t) d t, \;\;\; s \in [a, b]
\end{equation}
is a compact linear operator from $L^\infty [a, b] $ to $C[a, b].$ 
Assume that $1$ is in the 
resolvent set of $\mathcal{K}.$ Then for $ f \in C[a, b],$
the following Fredholm integral equation
\begin{equation}\label{eq:3.2}
x - \mathcal{K} x = f
\end{equation}
has a unique solution in $C[a, b]$, say $\varphi.$

In this special  case, the Nystr\"{o}m operator defined in (\ref{eq:2.8}) is given by
\begin{eqnarray}\label{eq:3.3}
(\mathcal{K}_m x)(s) &= &\tilde h \sum_{j=1}^m \sum_{i=1}^\rho w_i \; \kappa \left (s, \zeta_i^j \right ) \; x \left (\zeta_i^j \right ), \;\;\; s \in [a, b], \;\;\; x \in C [a, b].
\end{eqnarray}
Note that for $x \in C[a, b],$ $\mathcal{K}_m x$ converges to $\mathcal{K}x $ and $\{\mathcal{K}_m\}$ is a collectively compact 
family. Hence for all $m$ big enough, 
 $1$ is in the resolvent set of $\mathcal{K}_m$ and the following Nystr\"{o}m approximation 
\begin{equation}\nonumber
x_m - \mathcal{K}_m x_m = f
\end{equation}
has a unique solution, say $\varphi_m.$ 
Also, 
\begin{equation}\nonumber
\| ( I -  \mathcal{K}_m )^{-1} \| \leq C_2.
\end{equation}
If the kernel $\kappa \in C^{d} ([a, b] \times [a, b])$  and if the right hand side $f \in C^{d}  [a, b],$  then 
the exact solution $\varphi$ of (\ref{eq:3.2}) is in $C^{d} [a, b]$ and 
\begin{equation}\label{eq:3.6}
\|\varphi - \varphi_m \|_\infty = O \left ( \tilde{h}^{d} \right ).
\end{equation}
(See  Atkinson \cite{Atk3}).

\subsection{Discrete Modified Projection Methods}

Note that since $\mathcal{K}_m$ is a linear operator, the discrete modified projection operator defined by (\ref{eq:2.23}) can be written as
\begin{eqnarray}\nonumber
\tilde{\mathcal{K}}_n^M = Q_n \mathcal{K}_m + \mathcal{K}_m Q_n - Q_n \mathcal{K}_m Q_n
\end{eqnarray}
and
\begin{eqnarray}\nonumber
\|\mathcal{K}_m - \tilde{\mathcal{K}}_n^M \| & =&\| (I - Q_n)  \mathcal{K}_m   (I - Q_n)  \|.
\end{eqnarray}
Since $Q_n$ converges to the Identity operator pointwise on $C [a, b]$ and $\{ \mathcal{K}_m\}$ is a collectively compact family, it follows that
\begin{equation}\nonumber
\|(I - Q_n  ) \mathcal{K}_m \| \rightarrow 0 \; \mbox {as} \; n \rightarrow \infty.
\end{equation}
Hence using the estimate (\ref{eq:2.20}), we obtain 
\begin{eqnarray}\nonumber
\|\mathcal{K}_m - \tilde{\mathcal{K}}_n^M \| 
&\leq &(1 + q) \| (I - Q_n)  \mathcal{K}_m \| \rightarrow 0 \; \mbox {as} \; n \rightarrow \infty.
\end{eqnarray}

\newpage
Thus, for $n$ and $m$ big enough, 
 $1$ belongs to the resolvent set of $\tilde {\mathcal{K}}_n^M$ and
\begin{equation}\nonumber
x_n - \tilde {\mathcal{K}}_n^M x_n = f
\end{equation}
has a unique solution, say $z_n^M.$ Also,
\begin{eqnarray}\label{***}
 \| ( I -  \tilde{\mathcal{K}}_n^M )^{-1} \| \leq 2 C_2.
\end{eqnarray}
The Discrete Iterated Modified Projection solution is defined as
\begin{equation}\nonumber
\tilde{z}_n^M =  \mathcal{K}_m z_n^M + f.
\end{equation}




\subsection{Orders of Convergence}

Let $ p \in \N$ and choose $ m = n p.$ Then $ p \tilde {h} = h.$ Note that the $m \rho = n p \rho$ nodes in the composite quadrature rule (\ref{eq:2.6}) can be divided into $n$ groups of $p \rho$
nodes and each of  the $p \rho $  nodes in a group lie precisely in one of the subinterval of the partition (\ref{eq:2.15}). Based on this
observation, we
rewrite the expression (\ref{eq:3.3}) for the Nystr\"{o}m operator $\mathcal{K}_m$ as follows.
\begin{eqnarray}\label{eq:3.13}
(\mathcal{K}_m x)(s) & = &  \tilde h \sum_{k=1}^n \sum_{\nu = 1}^p   \sum_{i=1}^\rho w_i \; \kappa \left (s, \zeta_i^{(k - 1) p + \nu } \right ) \; x \left (\zeta_i^{(k - 1) p + \nu } \right ), \;\;\; s \in [a, b].
\end{eqnarray}
 Let
  \begin{equation}\nonumber
   \Psi (t) = (t - q_1) \cdots (t - q_r), \;\;\; t \in [0, 1],
   \end{equation}
where $q_1, \ldots, q_r$ are Gauss-Legendre zeros in $[0, 1].$ Then
  \begin{equation}\nonumber
  \int_0^1 t^j \; \Psi (t) d t =  \int_0^1 t^j (t - q_1) \cdots (t - q_r) = 0 \;\;\; \mbox{for} \;\;\; 0 \leq j \leq r - 1.
  \end{equation}
We prove below a discrete analogue of the above result when the integral is replaced by a composite  numerical quadrature. 
\begin{lemma}\label{lemma:3.1}
Let $\{w_i, i = 1, \ldots, \rho\}$
and $\{\mu_i, i = 1, \ldots, \rho\}$ be respectively the weights and the node points in 
the basic quadrature formula defined in (\ref{eq:2.4}). Then
\begin{equation}\label{eq:3.16}
\sum_{\nu = 1}^p \; \sum_{i=1}^\rho w_i \; \left (  \frac {\nu - 1 + \mu_i} {p} \right )^j \;  \Psi \left (\frac {\nu - 1 + \mu_i} {p} \right ) = 0,  \;\;\; 0 \leq j \leq r - 1.
\end{equation}
\end{lemma}
\begin{proof}

Note that
\begin{eqnarray*}
\int_0^1 f (t) d t = \sum_{\nu = 1}^p \int_{\frac {\nu - 1} {p}}^{\frac{\nu} {p}} f (t) d t
&\approx &  \frac {1} {p} \; \sum_{\nu = 1}^p \; \sum_{i=1}^\rho w_i \; f \left  ( \frac {\nu - 1 + \mu_i} {p} \right).
\end{eqnarray*}
The quadrature rule (\ref{eq:2.4}) is assumed to be exact for polynomials of degree $\leq 2 r - 1.$
Hence
$$  \int_0^1  t^j \; \Psi (t) \; d t =  \frac {1} {p} \; \sum_{\nu = 1}^p \; \sum_{i=1}^\rho w_i \; \left  ( \frac {\nu - 1 + \mu_i} {p} \right)^j \;  \Psi  \left  ( \frac {\nu - 1 + \mu_i} {p} \right),  \;\;\; 0 \leq j \leq r - 1.$$
Since for $0 \leq j \leq r - 1,$
$$   \int_0^1  t^j \; \Psi (t) \; d t = 0,  $$
the desired result follows.
\end{proof}

For future reference, we quote the following interpolation error estimates from Conte-de-Boor \cite {Conte}.
For $t \in [t_{k-1}, t_k],$
  \begin{eqnarray}\nonumber
  x (t) - (Q_n x) (t) &= & x [\tau_1^k, \ldots, \tau_r^k, t] \; (t - \tau_1^k) \cdots (t - \tau_r^k),
    \end{eqnarray}
     where $x [\tau_1^k, \ldots, \tau_r^k, t] $ denotes the divided difference of $x$
based on $\{\tau_1^k, \ldots, \tau_r^k, t \}.$

Substituting for $\tau_i^k$ from (\ref{eq:2.18}) we obtain,
\begin{eqnarray}\nonumber
  x (t) - (Q_n x) (t) 
  &= &x [\tau_1^k, \ldots, \tau_r^k, t] \;  \; (t - t_{k-1} - q_1 h) \cdots (t - t_{k-1} - q_r h)\\ \nonumber
  &= & x [\tau_1^k, \ldots, \tau_r^k, t] \;  \; \left (\frac{t - t_{k-1}} {h}  - q_1 \right ) \cdots \left (\frac{t - t_{k-1}} {h} - q_r \right) h^r \\
\nonumber
\\\label{eq:3.17}
    & = & x [\tau_1^k, \ldots, \tau_r^k, t] \; \;  \Psi \left ( \frac {t - t_{k-1}} {h} \right ) h^r.
  \end{eqnarray}
     If $x \in C^r [a, b],$  
   then
   \begin{eqnarray}\label{eq:3.18}
   \| x - Q_n x \|_\infty 
   &\leq & \frac {\|x^{(r)}\|_\infty } {r!} \|\Psi \|_\infty h^r
    =  C_3 \|x^{(r)} \|_\infty h^r,
   \end{eqnarray}
   where
   $$\displaystyle {C_3 = \frac {\|\Psi \|_\infty} {r!}}.$$

We use the following notation:
$$ \| \kappa \|_{2 r, \infty} = \max \{ \|D^{(i, j)} \kappa \|_\infty: 0 \leq i + j \leq 2 r \} \; \mbox{and} \;   \|x \|_{2 r , \infty } = \max \{ \| x^{(j)} \|_\infty: 0 \leq j \leq 2 r \}, $$
  where $x^{(j)}$ denotes the $j$th derivative of $x$ and 
 $$D^{(i, j)} \kappa (s, t) = \frac {\partial^{i + j} \kappa} {\partial s^i \; \partial t^j}  (s, t), \; \;\; a \leq s, t \leq b.$$
We now prove some preliminary results which are needed to obtain the orders of convergence in  the Discrete Modified 
ProjectionMethod  and the Discrete Iterated  Modifited 
Projection Method. The following proposition is crucial in what follows.

\begin{proposition}\label{prop:3.2}
If $\kappa \in C^r ( [a, b] \times [a, b]) $ and $ x \in C^{2 r} [a, b],$ then
\begin{eqnarray}\label{eq:3.19}
\| \mathcal{K}_m (I - Q_n) x \|_\infty 
& \leq & C_4 \| \kappa \|_{r, \infty}  \|x \|_{2 r, \infty } \; h^{2 r},
\end{eqnarray}
where
$$C_4 =  \frac {1} {r!} 2^r  (b - a)    \| \Psi \|_\infty \left (\sum_{i=1}^\rho |w_i| \right )$$
is a constant independent of $n$ and of $m.$
\end{proposition}
\begin{proof}
For $s \in [a, b],$
\begin{eqnarray}\nonumber
(\mathcal{K}_m (I - Q_n) x)(s) &=& \tilde h \sum_{k=1}^n \sum_{\nu = 1}^p   \sum_{i=1}^\rho w_i \; \kappa \left (s, \zeta_i^{(k - 1) p + \nu } \right ) \times \\\nonumber
&& \hspace* {1 in} 
\left [ x \left ( \zeta_i^{(k - 1) p + \nu } \right ) - Q_n x \left (\zeta_i^{(k - 1) p + \nu } \right ) \right ].
\end{eqnarray}
Substituting from   the relation (\ref{eq:3.17}), we obtain
\begin{eqnarray}\nonumber
(\mathcal{K}_m (I - Q_n) x)(s) & = & \tilde h \; h^{r} \sum_{k=1}^n \sum_{\nu = 1}^p   \sum_{i=1}^\rho w_i \; \kappa \left (s, \zeta_i^{(k - 1) p + \nu }  \right ) \times
 \\\nonumber
&& \hspace* {1 in} x  [\tau_1^k, \ldots, \tau_r^k, \zeta_i^{(k - 1) p + \nu }] \; \Psi \left ( \frac {\zeta_i^{(k - 1) p + \nu  }- t_{k-1}} {h} \right ).
\end{eqnarray}
Note that
\begin{equation}\label{eq:3.20}
 \frac {\zeta_i^{(k - 1) p + \nu  }- t_{k-1}} {h} = \frac {( \nu  - 1 + \mu_i ) \; \tilde h} {p \; \tilde h } = \frac { \nu  - 1 + \mu_i  } {p }.
 \end{equation}
Thus,
 \begin{eqnarray}\nonumber
(\mathcal{K}_m (I - Q_n) x)(s) & = & \tilde h \; h^{r} \sum_{k=1}^n \sum_{\nu = 1}^p   \sum_{i=1}^\rho w_i \; \kappa \left (s, \zeta_i^{(k - 1) p + \nu }  \right ) \times
 \\\nonumber
&& \hspace* {1 in} x  [\tau_1^k, \ldots, \tau_r^k, \zeta_i^{(k - 1) p + \nu }] \; \Psi \left ( \frac { \nu  - 1 + \mu_i  } {p } \right ).\\
\label{eq:3.21}
\end{eqnarray}
For a fixed $s \in [a, b],$ define
$$g_s^k (t) = \kappa (s, t) \; x [\tau_1^k, \ldots, \tau_r^k, t] \;\;\; \mbox{for} \;\;\; t_{k-1} \leq t \leq t_k.$$
Then
\begin{eqnarray}\label {eq:3.22}
(\mathcal{K}_m (I - Q_n) x)(s) &=& \tilde h \; h^{r} \sum_{k=1}^n \sum_{\nu = 1}^p   \sum_{i=1}^\rho w_i \; g_s^k \left ( \zeta_i^{(k - 1) p + \nu }  \right ) 
\; \Psi \left ( \frac { \nu  - 1 + \mu_i  } {p } \right ).
\end{eqnarray}
Since $k  \in C^r ([a, b] \times [a, b])$ and $x \in C^{2 r} [a, b],$ it follows that
$g_s^k \in C^r [t_{k-1}, t_k]$  for $ k = 1, \ldots, n.$
For $ t \in [t_{k-1}, t_k],$  we  expand $g_s^k (t)$ about $t_{k-1}$ and obtain
\begin{eqnarray}\nonumber
g_s^k (t) & = & \sum_{j = 0}^{r - 1} \frac {1} {j!} \; (t - t_{k-1})^j \; (g_s^k )^{(j)} (t_{k-1}) + \frac {1} {r!}  (t - t_{k-1})^r (g_s^k )^{(r)} (\xi_t^k),
\;\;\; \xi_t^k \in (t_{k-1}, t_k). 
\end{eqnarray}
Substitute the above expression for $g_s^k (t)$ in (\ref{eq:3.22}) to obtain
 \begin{eqnarray*}\nonumber
(\mathcal{K}_m (I - Q_n) x)(s) 
&=& \tilde h \; h^{r}  \sum_{k=1}^n  \sum_{j = 0}^{r - 1} \frac {1} {j!}  (g_s^k )^{(j)} (t_{k-1})  \times \\
&&  \hspace*{1 in} \sum_{\nu = 1}^p   \sum_{i=1}^\rho 
w_i  \left ( \zeta_i^{(k - 1) p + \nu} - t_{k-1} \right )^j  \; \Psi \left (  \frac { \nu  - 1 + \mu_i  } {p  } \right )\\
& + & \frac {1} {r!} \tilde h \; h^{r} \sum_{k=1}^n \sum_{\nu = 1}^p   \sum_{i=1}^\rho w_i \;  \left ( \zeta_i^{(k - 1) p + \nu} - t_{k-1} \right )^r (g_s^k )^{( r )} (\xi_{i, \nu}^k)  \Psi \left (  \frac { \nu  - 1 + \mu_i  } {p  } \right ). 
\end{eqnarray*}
Using (\ref{eq:3.20}) the above equation reduces to
\begin{eqnarray*}\nonumber
(\mathcal{K}_m (I - Q_n) x)(s) &=& \tilde h \; h^{r}  \sum_{k=1}^n  \sum_{j = 0}^{r - 1} \frac {1} {j!}  (g_s^k )^{(j)} (t_{k-1}) \; \tilde h^j \times \\
&&  \hspace*{1 in} \sum_{\nu = 1}^p   \sum_{i=1}^\rho 
w_i  \left (  \nu  - 1 + \mu_i   \right )^j  \; \Psi \left (  \frac { \nu  - 1 + \mu_i  } {p  } \right )\\
& + & \frac {1} {r!} \tilde h^{r+1} \; h^{r} \sum_{k=1}^n \sum_{\nu = 1}^p   \sum_{i=1}^\rho w_i \;  \left (  \nu  - 1 + \mu_i   \right )^r (g_s^k )^{( r )} (\xi_{i, \nu}^k) \Psi \left (  \frac { \nu  - 1 + \mu_i  } {p  } \right ).
\end{eqnarray*}
From Lemma 3.1 we see that the first term in the above expression vanishes. Hence 
\begin{eqnarray*}
\| \mathcal{K}_m (I - Q_n) x \|_\infty & \leq & \frac {1} {r!}  \tilde h^{r+1} \; h^{r}  n  \max_{s \in [a, b]} \max_{1 \leq k \leq n} 
\left ( \max_{ {t \in [t_{k-1}, t_k]}  } \left |  (g_s^k )^{(r)} (t) \right | \right ) \times\\
&& \hspace* {0.5 in}
\sum_{\nu = 1}^p   \sum_{i=1}^\rho   \left \{ \; |w_i| \;  \left | \nu  - 1 + \mu_i  \right |^r \left | \Psi \left (  \frac { \nu  - 1 + \mu_i  } {p  } \right )  \right |  \right \}.
\end{eqnarray*}
It can be seen that for $s \in [a, b],$ 
 \begin{equation}\nonumber
  \max_{ {t \in [t_{k-1}, t_k]}  } \left |  (g_s^k )^{(r)} (t) \right | \leq 2^r \| \kappa \|_{r, \infty} \|x \|_{2 r, \infty }.
  \end{equation}
Since $\mu_i \in [0, 1],$ it follows that
$$ \nu - 1 + \mu_i \leq p \;\;\; \mbox {for} \; \; \; 1 \leq \nu \leq p$$
and hence
\begin{eqnarray*}
\sum_{\nu = 1}^p   \sum_{i=1}^\rho   \left \{ \; |w_i| \;  \left | \nu  - 1 + \mu_i  \right |^r \left | \Psi \left (  \frac { \nu  - 1 + \mu_i  } {p  } \right )  \right |  \right \}
\leq p^{r+1} \|\Psi\|_\infty \left (\sum_{i=1}^\rho   |w_i| \right).
\end{eqnarray*}
Since $\tilde{h} p = h$ and $h n = b - a,$  it follows that
\begin{eqnarray}\nonumber
\| \mathcal{K}_m (I - Q_n) x \|_\infty &\leq& \ C_4 \|\kappa \|_{r, \infty} \; \|x \|_{2 r, \infty} \;  h^{2 r},
\end{eqnarray}
 which completes the proof.
\end{proof}
The following two results are proved using Proposition \ref{prop:3.2}.
\begin{proposition}\label{prop:3.3}
If $\kappa \in C^{2r} ( [a, b] \times [a, b]), $  then
\begin{eqnarray}\label{eq:3.25}
\| \mathcal{K}_m (I - Q_n) \mathcal{K}_m\|
& \leq & C_5 \| \kappa \|_{r, \infty}  \|\kappa  \|_{2 r, \infty } \; h^{2 r},
\end{eqnarray}
where
$$C_5 =  C_4  (b - a) \left (\sum_{i=1}^\rho   |w_i| \right)  = \frac {1} {r!} 2^r  (b - a)^2    \| \Psi \|_\infty \left (\sum_{i=1}^\rho |w_i| \right )^2$$
is a constant independent of $n$ and of $m.$
\end{proposition}

\begin{proof}
From Proposition \ref{prop:3.2}, we have
\begin{eqnarray}\nonumber
\| \mathcal{K}_m (I - Q_n)  \mathcal{K}_mx \|_\infty 
& \leq & C_4 \| \kappa \|_{r, \infty}  \| \mathcal{K}_m x \|_{2 r, \infty } \; h^{2 r}.
\end{eqnarray}
From (\ref{eq:3.3})
\begin{eqnarray}\nonumber
(\mathcal{K}_m x ) (s) = \tilde {h} \sum_{j=1}^m \sum_{i=1}^\rho w_i \kappa \left (s, \zeta_i^j \right ) x \left ( \zeta_i^j \right ).
\end{eqnarray}
Hence for $ \beta = 1, \ldots, 2 r,$
\begin{eqnarray}\nonumber
(\mathcal{K}_m x )^{(\beta)} (s) = \tilde {h} \sum_{j=1}^m \sum_{i=1}^\rho w_i  D^{(\beta, 0)} \kappa \left (s, \zeta_i^j \right ) x \left ( \zeta_i^j \right )
\end{eqnarray}
and 
\begin{eqnarray}\nonumber
\|(\mathcal{K}_m x )^{(\beta)}\|_\infty \leq \tilde {h} m \left (\sum_{i=1}^\rho | w_i | \right ) \left \| D^{(\beta, 0)} \kappa \right \|_\infty \|x \|_\infty.
\end{eqnarray}
As a consequence,
\begin{eqnarray}\nonumber
\| \mathcal{K}_m x \|_{2 r, \infty }  \leq (b - a) \left (\sum_{i=1}^\rho | w_i | \right ) \|\kappa \|_{2 r, \infty } \|x \|_\infty.
\end{eqnarray}
Thus,
\begin{eqnarray}\nonumber
\| \mathcal{K}_m (I - Q_n)  \mathcal{K}_mx \|_\infty 
& \leq & C_4 (b - a) \left (\sum_{i=1}^\rho | w_i | \right )  \| \kappa \|_{r, \infty} \|\kappa \|_{2 r, \infty } \|x \|_\infty \; h^{2 r}.
\end{eqnarray}
The desired estimate follows by taking the supremum over the set $\{ x \in C[a, b]: \|x \|_\infty \leq 1 \}.$
\end{proof}

\begin{proposition}\label{prop:3.4}
If $\kappa \in C^{2 r }( [a, b] \times [a, b]) $ and $ x \in C^{2 r} [a, b],$ then
\begin{equation}\label{eq:3.26}
\|(I - Q_n ) \mathcal{K}_m (I - Q_n) x \|_\infty \leq C_3 C_4 \| \kappa \|_{2 r, \infty}  \|x \|_{2 r, \infty } h^{3 r}.
\end{equation}
If $\kappa \in C^{3 r }( [a, b] \times [a, b]) $ and $ x \in C^{2 r} [a, b],$ then
\begin{equation}\label{eq:3.27}
\|  \mathcal{K}_m (I - Q_n ) \mathcal{K}_m (I - Q_n) x \|_\infty \leq  ( C_4)^2 \; \| \kappa \|_{ r, \infty}  \|\kappa \|_{3 r, \infty}  \|x\|_{2 r, \infty} h^{4 r}.
\end{equation}
\end{proposition}
\begin{proof}
 From the estimate (\ref{eq:3.18}), we obtain
\begin{equation*}
\|(I - Q_n ) \mathcal{K}_m (I - Q_n) x \|_\infty \leq C_3 \| \left ( \mathcal{K}_m (I - Q_n)x \right )^{(r)}\|_\infty h^r.
\end{equation*}
Differentiating  (\ref{eq:3.21}) $r$ times with respect to $s,$ we obtain
\begin{eqnarray}\nonumber
(\mathcal{K}_m (I - Q_n) x)^{(r)}(s) & = & \tilde h \; h^{r} \sum_{k=1}^n \sum_{\nu = 1}^p   \sum_{i=1}^\rho w_i \; \frac {\partial^r \kappa} {\partial s^r} \left (s, \zeta_i^{(k - 1) p + \nu }  \right ) \; 
 \times \\\nonumber
&& \hspace* {1 in} x  [\tau_1^k, \ldots, \tau_r^k, \zeta_i^{(k - 1) p + \nu }] \; \Psi \left ( \frac { \nu  - 1 + \mu_i  } {p } \right ).
\end{eqnarray}
Let 
$$\ell (s, t) = \frac {\partial^r \kappa} {\partial s^r} (s, t), \;\;\; a \leq s, t \leq b.$$
If $\kappa \in C^{2 r} ([a,b] \times [a, b]),$ then $\ell \in C^{ r} ([a,b] \times [a, b]).$
Hence proceeding  as in the proof of Proposition \ref{prop:3.2}, we obtain
\begin{eqnarray}\nonumber
\| (\mathcal{K}_m (I - Q_n) x)^{(r)}  \|_\infty &\leq&
C_4 \| \ell \|_{r, \infty}  \|x \|_{2 r, \infty } h^{2 r} \leq
 C_4 \| \kappa \|_{2 r, \infty}  \|x \|_{2 r, \infty } h^{2 r}.
\end{eqnarray}
As a consequence,
\begin{equation*}
\|(I - Q_n ) \mathcal{K}_m (I - Q_n) x \|_\infty \leq C_3 C_4 \| \kappa \|_{2 r, \infty}  \|x \|_{2 r, \infty } h^{3 r},
\end{equation*}
which completes the proof of (\ref{eq:3.26}).

Let $\kappa \in C^{3 r} ([a, b] \times [a, b])$ and $ x \in C^{2 r} [a, b].$ Then $ \mathcal{K}_m (I - Q_n) x \in C^{3 r} [a, b].$
Using (\ref{eq:3.19}) we obtain,
\begin{equation}\label{eq:3.28}
\|  \mathcal{K}_m (I - Q_n ) \mathcal{K}_m (I - Q_n) x \|_\infty \leq  C_4 \| \kappa \|_{ r, \infty}  \| \mathcal{K}_m (I - Q_n) x \|_{2 r, \infty } 
h^{2 r}.
\end{equation}
From (\ref{eq:3.21}) for $ j = 0, 1, \ldots, 2 r,$ 
\begin{eqnarray}\nonumber
(\mathcal{K}_m (I - Q_n) x)^{(j)}(s) & = & \tilde h \; h^{r} \sum_{k=1}^n \sum_{\nu = 1}^p   \sum_{i=1}^\rho w_i \; \frac {\partial^j \kappa} {\partial s^j}  \left (s, \zeta_i^{(k - 1) p + \nu }  \right ) \times
 \\\nonumber
&& \hspace* {1 in} x  [\tau_1^k, \ldots, \tau_r^k, \zeta_i^{(k - 1) p + \nu }] \; \Psi \left ( \frac { \nu  - 1 + \mu_i  } {p } \right ).
\end{eqnarray}
Note that $\displaystyle {\frac  {\partial^j \kappa} {\partial s^j} \in C^{3 r - j} ([a, b] \times [a, b]), }$ for $ j = 0, 1, \ldots, 2 r.$ 
Proceeding as in the proof of Proposition \ref{prop:3.2},  we obtain
\begin{eqnarray}\nonumber
\|(\mathcal{K}_m (I - Q_n) x)^{(j)} \|_\infty
& \leq & C_4  \;  \|\kappa \|_{j+r, \infty} \|x\|_{2 r, \infty} h^{2 r} 
\end{eqnarray}
and hence
\begin{eqnarray*}
\| \mathcal{K}_m (I - Q_n) x \|_{2 r, \infty }  &=& \max \left \{ \|(\mathcal{K}_m (I - Q_n) x)^{(j)} \|_\infty: 0 \leq j \leq 2 r \right \} \\
& \leq & C_4 \|\kappa \|_{3 r, \infty} \|x\|_{2 r, \infty} h^{2 r}.
\end{eqnarray*}
As a consequence,
\begin{equation*}
\|  \mathcal{K}_m (I - Q_n ) \mathcal{K}_m (I - Q_n) x \|_\infty \leq  ( C_4 )^2 \; \| \kappa \|_{ r, \infty}  \|\kappa \|_{3 r, \infty}  \|x\|_{2 r, \infty} h^{4 r},
\end{equation*}
which completes the proof of (\ref{eq:3.27}).
\end{proof}
We now prove our main result for linear integral equations.
\begin{theorem}\label{thm:1.5}
 If $\kappa \in C^{d } ([a, b] \times [a, b])$ and $ f \in C^{d} [a, b],$ then
\begin{equation}\label{eq:3.29}
\|\varphi - z_n^M \|_\infty = O \left ( \max \left \{ \tilde{h}^{d}, h^{3 r} \right \} \right ).
\end{equation}
If $\kappa \in C^{\max \{3 r, d \} } ([a, b] \times [a, b])$ and $ f \in C^{d} [a, b],$ then
\begin{equation}\label{eq:3.30}
\|\varphi - \tilde{z}_n^M \|_\infty = O \left ( \max \left \{ \tilde{h}^{d}, h^{4 r} \right \} \right ).
\end{equation}
\end{theorem}
\begin{proof}

Since
\begin{eqnarray}\nonumber
\varphi - z_n^M & =&  (I - \mathcal{K})^{-1} f -  (I - \tilde {\mathcal{K}}_n^M )^{-1} f 
 =   (I - \tilde {\mathcal{K}}_n^M)^{-1} (   \mathcal{K} - \tilde {\mathcal{K}}_n^M) \varphi,
\end{eqnarray}
using (\ref{***}) we obtain,
\begin{eqnarray}\nonumber
\|\varphi - z_n^M\|_\infty  & \leq& 2 C_2 \left  (\|   \mathcal{K}  \varphi - \mathcal{K}_m     \varphi \|_\infty +  \| ( \mathcal{K}_m - 
\tilde {\mathcal{K}}_n^M ) \varphi \|_\infty \right )\\\nonumber
& \leq & 2 C_2 \left  (\|   \mathcal{K}  \varphi - \mathcal{K}_m     \varphi \|_\infty + \| (I - Q_n) \mathcal{K}_m (I -   Q_n ) \varphi \|_\infty \right ).
\end{eqnarray}
Using the estimates (\ref{eq:2.11}) and (\ref{eq:3.26}) we obtain,
\begin{eqnarray}\nonumber
\|\varphi - z_n^M\|_\infty  & =&  O \left ( \max \left \{ \tilde{h}^{d}, h^{3 r} \right \} \right ),
\end{eqnarray}
which completes the proof of (\ref{eq:3.29}).

Note that
\begin{eqnarray}\nonumber
\varphi_m - z_n^M & =&  
   (I -  {\mathcal{K}}_m)^{-1} (   \mathcal{K}_m - \tilde {\mathcal{K}}_n^M) z_n^M\\\nonumber
& = & (I -  {\mathcal{K}}_m)^{-1}   (I - Q_n) \mathcal{K}_m (I - Q_n)   z_n^M .
\end{eqnarray}
Hence
\begin{eqnarray}\nonumber
\varphi_m - \tilde {z}_n^M & =&  \mathcal{K}_m  \varphi_m - \mathcal{K}_m z_n^M \\\nonumber
& = & (I -  {\mathcal{K}}_m)^{-1}   \mathcal{K}_m (I - Q_n) \mathcal{K}_m (I - Q_n)   z_n^M\\\nonumber
& = &  (I -  {\mathcal{K}}_m)^{-1}   \mathcal{K}_m (I - Q_n) \mathcal{K}_m (I - Q_n) (z_n^M - \varphi )
\\\nonumber
& + & (I -  {\mathcal{K}}_m)^{-1}   \mathcal{K}_m (I - Q_n) \mathcal{K}_m (I - Q_n) \varphi 
\end{eqnarray}
Thus,
\begin{eqnarray}\nonumber
\|\varphi_m - \tilde {z}_n^M \|_\infty & \leq & C_2 (1 + \|Q_n \| ) \|\mathcal{K}_m (I - Q_n) \mathcal{K}_m\| \|z_n^M - \varphi\|_\infty \\\nonumber
& + & C_2 \|\mathcal{K}_m (I - Q_n) \mathcal{K}_m (I - Q_n) \varphi \|_\infty.
\end{eqnarray}
Using  (\ref{eq:3.25}),  (\ref{eq:3.27}) and (\ref{eq:3.29}) and the fact that $d \geq 2 r,$ it can be seen that
$$ \|\varphi_m - \tilde {z}_n^M \|_\infty = O (h^{4 r}).$$
Hence using (\ref{eq:3.6}) we obtain,
\begin{eqnarray*}
\|\varphi - \tilde {z}_n^M \|_\infty  \leq \|\varphi - \varphi_m \|_\infty + \|\varphi_m - \tilde {z}_n^M \|_\infty = 
 O \left ( \max \left \{ \tilde{h}^{d}, h^{4 r} \right \} \right ),
\end{eqnarray*}
which completes the proof.
\end{proof}


\setcounter{equation}{0}
\section {Discrete Modified Projection Method for Urysohn Integral Equations}

In this section we consider approximation of the Urysohn integral equation (\ref{eq:1.1})-(\ref{eq:1.2}) by the discrete version 
of the modified projection  method.  For a fixed $\delta > 0,$ a closed neighbourhood $B (\varphi, \delta)$ of the exact solution $\varphi$ is defined in 
(\ref{eq:2.13}). First we prove a result about the Nystr\"{o}m operator $\mathcal{K}_m$
defined in (\ref{eq:2.8}).
Let  $\displaystyle { \frac {\partial^2 \kappa} {\partial u^2} \in C (\Omega)}$ and define
\begin{equation} \label{eq:4.1}
C_{6} = \max_{\stackrel {s, t \in [a, b]}{|u| \leq \|\varphi \|_\infty + \delta }}
\left | \frac {\partial^{2} \kappa } {\partial u^2} (s, t, u) \right |.
\end{equation}

\begin{proposition}\label{prop:4.1}
Let $\displaystyle { \frac {\partial^2 \kappa} {\partial u^2} \in C (\Omega).}$
Then for $v_1, v_2 \in B (\varphi, \delta)$ and for $s \in [a, b],$
\begin{eqnarray}\nonumber
\mathcal{K}_m (  v_2 ) (s) -  \mathcal{K}_m ( v_1 ) (s) - \mathcal{K}_m' ( v_1) (v_2 - v_1) (s) &= & 
 R  (v_2 - v_1)  (s), 
 \end{eqnarray} 
where
\begin{eqnarray}\nonumber
\| R  (v_2 - v_1) \|_\infty \leq  \frac {C_{6} (b - a)} {2}  \left ( \sum_{i=1}^\rho | w_i | \right ) \|v_2 - v_1\|_\infty^2.
\end{eqnarray}
\end{proposition}
\begin{proof}
If $v_1, v_2 \in B (\varphi, \delta),$ then by the generalized Taylor's theorem,
\begin{eqnarray}\nonumber
\mathcal{K}_m (  v_2 ) (s) -  \mathcal{K}_m ( v_1 ) (s) - \mathcal{K}_m' ( v_1) (v_2 - v_1) (s) &= & 
 R  (v_2 - v_1)  (s), \; s \in [a, b], 
 \end{eqnarray} 
 where
 \begin{eqnarray}\label{eq:4.3}
&&R  (v_2 - v_1)  (s)
 = \int_0^1  {(1 - \theta) }   \mathcal {K}_m^{''} \left (v_1 + \theta  
 (v_2 - v_1) \right ) (v_2 - v_1)^2  (s)    d \theta.
\end{eqnarray}
 For $s \in [a, b]$ and $ \theta \in [0, 1],$ define
\begin{eqnarray}\nonumber
(S_\theta  (v_2 - v_1) ) (s) & =&  \mathcal {K}_m^{''} \left (v_1 + \theta
 (v_2 - v_1) \right ) (v_2 - v_1)^2 (s)\\\nonumber
& = & \tilde {h}  \sum_{j=1}^m  \sum_{i=1}^\rho  w_i \;  \frac {\partial^2 \kappa } {\partial u^2} \left (s,  \zeta_i^j, 
v_1  (\zeta_i^j) + \theta
 (v_2 - v_1)  (\zeta_i^j) \right ) (v_2 - v_1)^2 (\zeta_i^j).\\\label{eq:4.4}
\end{eqnarray}
Then 
\begin{eqnarray}\nonumber
\|S_\theta  (v_2 - v_1) \|_\infty \leq C_{6 } (b - a) \left ( \sum_{i=1}^\rho | w_i | \right )  \|v_2 - v_1\|_\infty^2.
\end{eqnarray}
Since
 \begin{eqnarray}\nonumber
R   (v_2 - v_1)  (s) 
&=& \int_0^1    (1 - \theta)   S_\theta  (v_2 - v_1)  (s)  d \theta,
\end{eqnarray}
it follows that
\begin{equation}\nonumber
\| R   (v_2 - v_1) \|_\infty \leq  \frac {C_{6} (b - a)} {2}  \left ( \sum_{i=1}^\rho | w_i | \right )  \|v_2 - v_1\|_\infty^2.
\end{equation}
This completes the proof. \end{proof}


\begin{remark}
 Note that
\begin{eqnarray}
\mathcal{K} ' (\varphi ) v (s) =  \int_a^b \frac {\partial \kappa } {\partial u} (s, t, \varphi (t)) v (t) \; d t,
\end{eqnarray}
whereas
\begin{equation}\label{****}
\mathcal {K}_m'  (\varphi) v (s)  =  \tilde {h}  \sum_{j=1}^m  \sum_{i=1}^\rho  w_i \;  \frac {\partial \kappa } {\partial u} (s,  \zeta_i^j, \varphi (\zeta_i^j)) v (\zeta_i^j), \;\;\; s \in [a, b].
\end{equation}
Thus,
$\mathcal {K}_m'  (\varphi)$  is the Nystr\"{o}m approximation of the linear operator $\mathcal{K} ' (\varphi ) $  associated with a convergent quadrature formula. Hence
$\mathcal {K}_m'  (\varphi)$ converges to $\mathcal{K} ' (\varphi ) $ pointwise and $\mathcal{K}_m ' (\varphi ) $ is a collectively compact family.

As before, now onwards  we assume that 
$$ m = n p \;\;\; \mbox { for some} \;\;\;  p \in \N.$$
It follows that
\begin{eqnarray}\label{*}
 \|(I - Q_n) \mathcal{K}'_m (\varphi) \| \rightarrow 0\; \mbox{as} \; n \rightarrow \infty.
\end{eqnarray}
It is assumed that 
$ (I - \mathcal{K}' (\varphi ))^{-1}: C[a, b] \rightarrow C [a, b]$  is a bounded linear operator. 
Hence
there exists a positive integer $m_1 \geq m_0$ such that for $m \geq m_1,$
$ (I - \mathcal{K}'_m (\varphi ))^{-1} $ exists and
\begin{equation}\label{eq:4.6}
 \left \|(I - \mathcal{K}'_m (\varphi ))^{-1} \right \| \leq C_{7}. 
\end{equation}
See Atkinson \cite {Atk3}.
\end{remark}

We prove some preliminary results which are needed to obtain the order of convergence of the discrete modified projection solution $z_n^M.$
\begin{proposition}\label{prop:4.2}
Let $\displaystyle { \frac {\partial^2 \kappa} {\partial u^2} \in C (\Omega).}$ Then $\mathcal{K}'_m$ is Lipschitz continuous in  $B (\varphi, \delta):$
\begin{equation}\label{eq:4.7}
 \|\mathcal{K}'_m (x) -  \mathcal{K}'_m (y) \| \leq \gamma\| x - y \|_\infty , \; \; \;  x, y \in B (\varphi, \delta),
\end{equation}
where 
$\gamma$ is a constant independent of $m.$
\end{proposition}

\begin{proof}
For $ x , y \in B (\varphi, \delta), $
\begin{eqnarray*}
\|  \mathcal{K}'_m (x) -  \mathcal{K}'_m (y) \| 
& = & \sup_{\|v\| \leq 1 } \sup_{s \in [a, b] } \left | \left ( \mathcal{K}'_m (x) -  \mathcal{K}'_m (y) \right ) v (s) \right |
\end{eqnarray*}
For $s \in [a, b],$ we have
\begin{eqnarray*}
\left ( \mathcal{K}'_m (x) -  \mathcal{K}'_m (y) \right ) v (s) & = & \tilde {h}  \sum_{j=1}^m  \sum_{i=1}^\rho  w_i \;  
\left ( \frac {\partial \kappa } {\partial u} (s,  \zeta_i^j, x (\zeta_i^j)) - \frac {\partial \kappa } {\partial u} (s,  \zeta_i^j, y (\zeta_i^j)) \right ) 
v (\zeta_i^j).
\end{eqnarray*}
By the Mean Value Theorem,
\begin{eqnarray*}
\left ( \mathcal{K}'_m (x) -  \mathcal{K}'_m (y) \right ) v (s) 
& = & \tilde {h}  \sum_{j=1}^m  \sum_{i=1}^\rho  w_i \;  
\frac {\partial^2 \kappa } {\partial u^2} (s,  \zeta_i^j, \eta_i^j) \left (x (\zeta_i^j) - y (\zeta_i^j) \right)
v (\zeta_i^j),
\end{eqnarray*}
where $$\eta_i^j = \theta_i^j x (\zeta_i^j) + (1 - \theta_i^j ) y (\zeta_i^j) \; \mbox { for some} \;  \theta_i^j  \in (0, 1).$$
Since  $ x , y \in B (\varphi, \delta), $ it follows that 
$$|\eta_i^j | \leq \|\varphi \|_\infty + \delta, \;\;\;  i = 1, \ldots, \rho, \; j = 1, \ldots, m.$$
Hence for $ s \in [a, b],$
\begin{eqnarray*}
\left | \left ( \mathcal{K}'_m (x) -  \mathcal{K}'_m (y) \right ) v (s) \right | &\leq&  C_{6}  \tilde {h} m  \left ( \sum_{i=1}^\rho  |w_i |  \right ) 
\|x - y \|_\infty \|v\|_\infty \\
&=&  \gamma \|x - y \|_\infty \|v\|_\infty,
\end{eqnarray*}
where
$$\gamma = C_{6} (b - a)  \left ( \sum_{i=1}^\rho  |w_i |  \right ), $$
 $C_{6}$ being defined in (\ref{eq:4.1}).
Hence
\begin{eqnarray}\nonumber
 \|\left ( \mathcal{K}'_m (x) -  \mathcal{K}'_m (y) \right ) v \|_\infty  \leq \gamma\| x - y \|_\infty \|v\|_\infty. 
\end{eqnarray}
Taking the supremum over the set $\{ v \in C[a, b]: \|v\|_\infty \leq 1 \},$
the required result follows.
\end{proof}

\begin{proposition}\label{prop:4.3}
Let $\displaystyle { \frac {\partial^2 \kappa} {\partial u^2} \in C (\Omega).}$  There exists a positive integer $n_1$ such that for $n \geq n_1$ and for $ m \geq m_1,$ $I - \left (\tilde{\mathcal{K}}_n^M \right )' (\varphi)$ is invertible
and
\begin{eqnarray}\label{eq:4.8}
\left \|\left (  I - \left (\tilde{\mathcal{K}}_n^M \right )'  (\varphi) \right )^{-1} \right \| \leq 2 \; C_{7}.
\end{eqnarray}
\end{proposition}

\begin{proof}

\noindent
Fix $m = n p \geq m_1.$ Then by (\ref{eq:4.6})
$$ \left \| (I - \mathcal{K}_m' (\varphi ))^{-1} \right \| \leq C_{7}. $$
From the definition of $\tilde{\mathcal{K}}_n^M (x)$ in (\ref{eq:2.23}), it follows that
\begin{equation}\nonumber
\left (\tilde{\mathcal{K}}_n^M \right )' (x) = Q_n  \mathcal{K}_m' (x) + (I - Q_n)  \mathcal{K}_m' (Q_n x) Q_n.
\end{equation}
 Hence
\begin{eqnarray*}
 \mathcal{K}_m' (\varphi ) -  \left (\tilde{\mathcal{K}}_n^M \right )' (\varphi)& = &  (I - Q_n) \mathcal{K}_m' (\varphi ) - (I - Q_n) \mathcal{K}_m' (Q_n \varphi ) Q_n\\
& = & (I - Q_n) \mathcal{K}_m' (\varphi ) (I - Q_n) + (I - Q_n) (\mathcal{K}_m' (\varphi) - \mathcal{K}_m' (Q_n \varphi ) )Q_n.
\end{eqnarray*}
Since $\varphi \in C [a, b],$ it follows that 
$$\|Q_n \varphi - \varphi \|_\infty \rightarrow 0 \; \mbox {as} \; n \rightarrow \infty.$$
Hence there exists a positive integer $n_0$ such that
$$ n \geq n_0 \Rightarrow Q_n \varphi \in B (\varphi, \delta).$$
By Proposition \ref{prop:4.2},
$$ \|\mathcal{K}_m' (\varphi) - \mathcal{K}_m' (Q_n \varphi )\| \leq \gamma \|\varphi - Q_n \varphi\|_\infty \rightarrow 0$$
and by (\ref{*})
$$ \|(I - Q_n) \mathcal{K}_m' (\varphi )\| \rightarrow 0 \; \mbox{as} \; n \rightarrow \infty.$$
Hence for $ n \geq n_0,$
\begin{eqnarray*}
 \left \| \mathcal{K}_m' (\varphi ) -  \left (\tilde{\mathcal{K}}_n^M \right )' (\varphi) \right \| &\leq & 
( 1 + \|Q_n\|)  \|(I - Q_n) \mathcal{K}_m' (\varphi )\| \\
&+& \| Q_n \| (1 + \|Q_n \|)  \|\mathcal{K}_m' (\varphi) - \mathcal{K}_m' (Q_n \varphi )\|\\
&\leq & (1 + q) \|(I - Q_n) \mathcal{K}_m' (\varphi )\| + q (1 + q) \gamma \|\varphi - Q_n \varphi\|_\infty.  
\end{eqnarray*}
Thus,
\begin{eqnarray*}
 \left \| \mathcal{K}_m' (\varphi ) -  \left (\tilde{\mathcal{K}}_n^M \right )' (\varphi) \right \| 
& \rightarrow & 0 \; \mbox{as} \; n \rightarrow \infty.
\end{eqnarray*}
Choose $n_1 \geq n_0$ such that
$$ n \geq n_1 \Rightarrow \left \| \left (\tilde{\mathcal{K}}_n^M \right )'  (\varphi) - \mathcal{K}_m' (\varphi ) \right \| 
\leq 
\frac {1} {2 \; C_{7}}.$$
Since
\begin{eqnarray*}
 I - \left (\tilde{\mathcal{K}}_n^M \right )'  (\varphi) 
&=& \left [I - \left \{ \left (\tilde{\mathcal{K}}_n^M \right )'  (\varphi) - \mathcal{K}_m' (\varphi )\right \}  (I - \mathcal{K}_m' (\varphi ))^{-1} \right ]
 (I - \mathcal{K}_m' (\varphi )),
\end{eqnarray*}
it follows that  for $ n \geq n_1,$ 
$$ \left \| \left (  I - \left (\tilde{\mathcal{K}}_n^M \right )'  (\varphi) \right )^{-1} \right \| \leq 2 \;
\left \| (I - \mathcal{K}_m' (\varphi ))^{-1} \right \| \leq 2 \; C_{7}.$$
This completes the proof.
\end{proof}
\begin{remark}\label{remark:4.4}

For $ n \geq n_1$ and $m \geq m_1,$ define
\begin{eqnarray}\label{eq:4.10}
B_n (x) = x - \left [ I -\left (\tilde{\mathcal{K}}_n^M \right )'  (\varphi) \right]^{-1}  \left \{  x  - \tilde{\mathcal{K}}_n^M  ( x) - f  \right \}
\end{eqnarray}
Then
\begin{eqnarray}\label{eq:4.11}
B_n (x) = x \;\;\; \mbox{if and only if} \;\;\;  x  -\tilde{\mathcal{K}}_n^M  ( x) = f.
\end{eqnarray}
As in Grammont \cite{Gram2}, it can be shown that there is a $\delta_0 > 0$ such that $B_n$ has a unique 
fixed point $z_n^M$ in $B (\varphi, \delta_0)$ and that 
\begin{eqnarray*}
  \| z_n^M - \varphi\|_\infty 
 \leq \frac {3} {2}  \left \| \left [ I -\left (\tilde{\mathcal{K}}_n^M \right )'  (\varphi) \right]^{-1}  \left [ \mathcal{K} (\varphi) - \tilde{\mathcal{K}}_n^M (\varphi) \right ] \right \|_\infty.
\end{eqnarray*}
Hence
\begin{eqnarray}\label{eq:4.12}
\| z_n^M - \varphi\|_\infty 
& \leq & 3 \; C_{ 7 } \;  \left (  \|\mathcal{K} (\varphi) - 
\mathcal{K}_m (\varphi)\|_\infty + \|\mathcal{K}_m (\varphi) -\tilde{\mathcal{K}}_n^M (\varphi)\|_\infty \right ).
\end{eqnarray}
Without loss of generality, assume that 
$$ n \geq n_1 \Rightarrow  Q_n \varphi \in B (\varphi, \delta_0) \; \mbox {and} \; n \geq n_1, m \geq m_1 = n_1 p
\Rightarrow \varphi_m \in B (\varphi, \delta_0), Q_n \varphi_m \in B (\varphi, \delta_0).$$
By (\ref{eq:2.11}),
$$\|\mathcal{K} (\varphi) - \mathcal{K}_m (\varphi)\|_\infty = O \left (\tilde{h}^d \right ).$$
In order to obtain the order of convergence for the term $\|\mathcal{K}_m (\varphi) -\tilde{\mathcal{K}}_n^M (\varphi)\|_\infty$ in the  estimate (\ref{eq:4.12}), we prove the following result.
\end{remark}

\begin{proposition}\label{prop:4.5}
Let $ r \geq 1. $ 
If $\displaystyle {\frac {\partial^2 \kappa} {\partial u^2} \in C^{ r} (\Omega)}$ and $ f \in C^{ r} [a, b],$ then  for $ n \geq n_1,$
\begin{eqnarray}\label{eq:4.14}
\|(I - Q_n)  (\mathcal{K}_m (Q_n\varphi) - \mathcal{K}_m (\varphi) - \mathcal {K}_m'  (\varphi) (Q_n \varphi - \varphi ) ) \|_\infty
= O (h^{3 r}).
\end{eqnarray}
\end{proposition}
\begin{proof}
Let $\delta = \delta_0,$ $v_1 = \varphi$ and $v_2 = Q_n \varphi$ in Proposition \ref{prop:4.1}.
Then
\begin{eqnarray*}
&&  \mathcal{K}_m (Q_n\varphi) - \mathcal{K}_m (\varphi) - \mathcal {K}_m'  (\varphi) (Q_n \varphi - \varphi )  = R  ( Q_n \varphi - \varphi ).
\end{eqnarray*}
By (\ref{eq:3.18})
$$\|(I - Q_n) R ( Q_n \varphi - \varphi ) \|_\infty  \leq C_3 \|( R  ( Q_n \varphi - \varphi ) )^{(r)}\|_\infty h^r.$$
We have
\begin{eqnarray}\nonumber
( R  ( Q_n \varphi - \varphi ) )^{(r)} (s) 
&=&  \int_0^1    (1 - \theta)  (S_\theta ( Q_n \varphi - \varphi ) )^{(r)} (s)   d \theta,
\end{eqnarray}
where from (\ref{eq:4.4}),
\begin{eqnarray*}
&&(S_\theta ( Q_n \varphi - \varphi ) )^{(r)} (s) \\
&  & \hspace*{0.5 in} = \tilde {h}  \sum_{j=1}^m  \sum_{i=1}^\rho  w_i \;  \frac {\partial^{r +2} \kappa } {\partial s^r\partial u^2} \left (s,  \zeta_i^j, 
\varphi (\zeta_i^j)+ \theta (Q_n \varphi - \varphi )    (\zeta_i^j) \right ) ( Q_n \varphi - \varphi )^2 (\zeta_i^j)
\end{eqnarray*}
Let
$$ C_{8} = \max_{\stackrel {s, t \in [a, b]}{|u| \leq \|\varphi \|_\infty + \delta_0 }}
\left | \frac {\partial^{r +2} \kappa } {\partial s^r\partial u^2} (s, t, u) \right |.$$
Then
$$\|(S_\theta ( Q_n \varphi - \varphi ) )^{(r)}\|_\infty \leq C_{8} (b - a) \left ( \sum_{i=1}^\rho | w_i | \right ) \|Q_n \varphi - \varphi \|_\infty^2$$
and
\begin{equation}\nonumber
\|( R  ( Q_n \varphi - \varphi ) )^{(r)}\|_\infty \leq  \frac {C_{8} (b - a)} {2}  \left ( \sum_{i=1}^\rho | w_i | \right ) \|Q_n \varphi - \varphi \|_\infty^2.
\end{equation}
Since $ \kappa \in C^r (\Omega)$ and $ f \in C^{ r} [a, b],$ it follows that that $\varphi \in C^r [a, b].$ Hence by  (\ref{eq:3.18}),
\begin{equation}\nonumber
\|Q_n \varphi - \varphi \|_\infty \leq C_3 \|\varphi^{(r)}\|_\infty h^r.
\end{equation}
Thus,
\begin{eqnarray*}
\|(I - Q_n) R ( Q_n \varphi - \varphi ) \|_\infty 
& \leq &  \frac {( C_3)^3 C_{8} (b - a)} {2}  \left ( \sum_{i=1}^\rho | w_i | \right ) \| \varphi^{(r)}\|_\infty^2 h^{3 r},
\end{eqnarray*}
which completes the proof of  (\ref{eq:4.14}).
\end{proof}



In the following theorem we obtain the order of convergence in the discrete modified projection method.


\begin{theorem}\label{thm:4.6}
Let $ r \geq 1, $ $ \kappa \in C^d (\Omega), $ $\displaystyle {\frac {\partial \kappa} {\partial u} \in C^{2 r} (\Omega)}$ and $ f \in C^{d} [a, b].$  
Let $\varphi$ be the unique solution of (\ref{eq:2.2}) and assume that $1$ is not an eigenvalue of $\mathcal{K}' (\varphi).$ 
Let $ n \geq n_1$ and $m \geq m_1.$
Let $\mathcal{X}_n$ be the space of piecewise polynomials of degree $\leq r - 1$ with respect to the partition (\ref{eq:2.15})
and $Q_n: L^\infty [0, 1] \rightarrow \mathcal {X}_n$ be the interpolatory projection at $r$ Gauss points defined by
(\ref{eq:2.19}). Let $z_n^M $ be the unique solution of (\ref{eq:2.24}) in $B (\varphi, \delta_0).$ Then
\begin{equation}\label{eq:4.15}
\|z_n^M - \varphi \|_\infty = O (\max \{\tilde{h}^{d}, h^{3 r} \}).
\end{equation}
\end{theorem}
\begin{proof}

From (\ref{eq:4.12}),
\begin{eqnarray}\label{eq:4.16}
\| z_n^M - \varphi\|_\infty 
& \leq & 3 \; C_{7 } \;   \left ( \|\mathcal{K} (\varphi) - 
\mathcal{K}_m (\varphi)\|_\infty + \|\mathcal{K}_m (\varphi) -\tilde{\mathcal{K}}_n^M (\varphi)\|_\infty \right ).
\end{eqnarray}
From (\ref{eq:2.11}),
\begin{eqnarray}\label{eq:4.17}
\|\mathcal{K} (\varphi) -  \mathcal{K}_m (\varphi)\|_\infty = O \left ( \tilde{h}^{d} \right ).
\end{eqnarray}
Note that
\begin{eqnarray}\nonumber
\mathcal{K}_m (\varphi) - \tilde{\mathcal{K}}_n^M (\varphi)  &= & (I - Q_n) (\mathcal{K}_m (\varphi) - \mathcal{K}_m (Q_n\varphi))\\\nonumber
& = & - (I - Q_n) (\mathcal{K}_m (Q_n\varphi) - \mathcal{K}_m (\varphi) - \mathcal {K}_m'  (\varphi) (Q_n \varphi - \varphi ) )\\\nonumber
&& -  (I - Q_n) \mathcal {K}_m'  (\varphi) (Q_n \varphi - \varphi ). 
\end{eqnarray}
Hence
\begin{eqnarray}\nonumber
\|\mathcal{K}_m (\varphi) - \tilde{\mathcal{K}}_n^M (\varphi)  \|_\infty  
& \leq & \| (I - Q_n) (\mathcal{K}_m (Q_n\varphi) - \mathcal{K}_m (\varphi) - \mathcal {K}_m'  (\varphi) (Q_n \varphi - \varphi ) )\|_\infty\\\label{eq:4.18}
&& + \| (I - Q_n) \mathcal {K}_m'  (\varphi) (Q_n \varphi - \varphi )\|_\infty. 
\end{eqnarray}
By (\ref{eq:4.14}) of Proposition \ref{prop:4.5},
\begin{eqnarray}\label{eq:4.19}
\| (I - Q_n) (\mathcal{K}_m (Q_n\varphi) - \mathcal{K}_m (\varphi) - \mathcal {K}_m'  (\varphi) (Q_n \varphi - \varphi ) )\|_\infty = O (h^{3 r}).
\end{eqnarray}
Define
$$\ell (s, t) = \frac {\partial \kappa} {\partial u} (s, t, \varphi (t)).$$
Then from (\ref{****}),
\begin{equation}\nonumber
\mathcal {K}_m'  (\varphi) v (s)  =  \tilde {h}  \sum_{j=1}^m  \sum_{i=1}^\rho  w_i \;  \ell (s,  \zeta_i^j) v (\zeta_i^j), \;\;\; s \in [a, b].
\end{equation}
By assumption, $\displaystyle {\frac {\partial \kappa} {\partial u} \in C^{2r} (\Omega)}.$ Since $f \in C^{2r} [a, b],$
it follows that $\varphi \in  C^{2r} [a, b].$ Hence \\ $\ell \in C^{2 r} ([a, b] \times [a, b]).$ Thus,
Proposition \ref{prop:3.4} is applicable and we obtain
\begin{eqnarray}\label{eq:4.20}
\| (I - Q_n) \mathcal {K}_m'  (\varphi) (Q_n \varphi - \varphi )\|_\infty = O (h^{3 r}).
\end{eqnarray}
From (\ref{eq:4.18})-(\ref{eq:4.20}), it follows that
\begin{eqnarray}\label{eq:4.21}
\|\mathcal{K}_m (\varphi) - \tilde{\mathcal{K}}_n^M (\varphi)  \|_\infty = O (h^{3 r}).
\end{eqnarray}
The required result follows from (\ref{eq:4.16}), (\ref{eq:4.17}) and  (\ref{eq:4.21}).
\end{proof}


\setcounter{equation}{0}
\section{Discrete Iterated Modified Projection method for Urysohn Integral Equations}

Recall from Remark \ref{remark:4.4} that there exists $\delta_0 > 0$ and a positive integer $n_1$ such that
for $n \geq n_1$  equation (\ref{eq:2.24}) has a unique solution $z_n^M$ in $ B (\varphi, \delta_0).$ 

 The  discrete iterated modified projection solution is defined as 
\begin{equation}\label{eq:5.1}
\tilde{z}_n^M =   \mathcal {K}_m ({z}_n^M) + f.
\end{equation}
In this section we show that $\tilde{z}_n^M \rightarrow \varphi$ as $n \rightarrow \infty $ and  obtain its order of convergence.

\begin{proposition}\label{prop:5.1}
Let $ r \geq 1, $  $ \kappa \in C^d (\Omega), $ $\displaystyle {\frac {\partial \kappa} {\partial u} \in C^{2 r} (\Omega)}$ and $ f \in C^{d} [a, b].$ 
Let $\varphi_m$ denote the Nystr\"{o}m approximation of the exact solution $\varphi$ of the Urysohn integral equation (\ref{eq:1.1})-(\ref{eq:1.2}).
Then for $n \geq n_1$ and $m = n p,$ 
\begin{eqnarray}\label{eq:5.2}
\tilde{z}_n^M - \varphi_m =  \mathcal {K}_m' (\varphi_m) (z_n^M - \varphi_m) + O (\max \{\tilde{h}^{d}, h^{3 r} \}^2).
\end{eqnarray}
\end{proposition}
\begin{proof}
Since
$$ \varphi_m - \mathcal {K}_m (\varphi_m) = f,$$
it follows that
\begin{equation}\label{eq:5.3}
\tilde{z}_n^M - \varphi_m =  \mathcal {K}_m ({z}_n^M)  - \mathcal {K}_m (\varphi_m).
\end{equation}
In Theorem \ref{thm:4.6} we proved that
\begin{equation}\nonumber
\|z_n^M - \varphi \|_\infty = O (\max \{\tilde{h}^{d}, h^{3 r} \}).
\end{equation}
 By  Proposition \ref{prop:4.1} with $\delta = \delta_0,$  for $ s \in [a, b],$
\begin{eqnarray}\label{eq:5.4}
\mathcal {K}_m ({z}_n^M) (s) - \mathcal {K}_m (\varphi_m) (s) = \mathcal {K}_m' (\varphi_m) (z_n^M - \varphi_m)  (s) + 
R (z_n^M - \varphi_m) (s),
\end{eqnarray}
where
\begin{eqnarray}\nonumber
\|R (z_n^M - \varphi_m)\|_\infty & \leq & \frac {C_{6} (b - a)} {2} \left (\sum_{i=1}^\rho |w_i| \right ) \|z_n^M - \varphi_m\|_\infty^2
\\\label{eq:5.5}
& = & O (\max \{\tilde{h}^{d}, h^{3 r} \}^2).
\end{eqnarray}
The required result follows from (\ref{eq:5.3})-(\ref{eq:5.5}).
\end{proof}
In order to obtain an error estimate for the first term in (\ref{eq:5.2}) we need to define a new operator which has $z_n^M$ 
as a fixed point. For this purpose, we show that for all $m$ large enough, $ I - \mathcal{K}_m' (\varphi_m)$ are invertible and are
uniformly bounded.
\begin{proposition}\label{prop:5.2}
Let $\displaystyle { \frac {\partial^2 \kappa} {\partial u^2} \in C (\Omega).}$  There exists a positive  integer $m_2 \geq m_1$ such that for $m \geq m_2$
\begin{equation}\label{eq:5.6}
\| (I - \mathcal{K}_m' (\varphi_m))^{-1}\| \leq 2 \; C_{7}.
\end{equation}
\end{proposition}
\begin{proof}
Recall from (\ref{eq:4.6}) that for $m \geq m_1,$
$$\| (I - \mathcal{K}_m' (\varphi))^{-1}\| \leq   C_{7}.$$
By Proposition \ref{prop:4.2} with $\delta = \delta_0,$
\begin{eqnarray*}
\left  \|\mathcal {K}_m'  (\varphi)  - \mathcal {K}_m'  (\varphi_m) \right \| \leq \gamma \|\varphi - \varphi_m \|_\infty \rightarrow 0 \; \mbox {as} \; m \rightarrow \infty.  
\end{eqnarray*}
Choose $m_2 \geq m_1$ such that
$$  m \geq m_2 \Rightarrow \|\varphi - \varphi_m \|_\infty   \leq \frac {1} {2 C_{7} \gamma}.$$
Since
\begin{eqnarray*}
I - \mathcal{K}_m' (\varphi_m) 
& = & (I - (\mathcal{K}_m' (\varphi_m)  - \mathcal{K}_m' (\varphi)) (I - \mathcal{K}_m' (\varphi))^{-1}) (I - \mathcal{K}_m' (\varphi)),
\end{eqnarray*}
it follows that for $m \geq m_2,$
$$  (I - \mathcal{K}_m' (\varphi_m))^{-1} \; \mbox{exists} $$
and that 
\begin{equation}\nonumber
\| (I - \mathcal{K}_m' (\varphi_m))^{-1}\| \leq 2 \; C_{7}.
\end{equation}
This completes the proof.
\end{proof}

\begin{remark}
In (\ref{eq:4.10}) we  defined an operator $B_n$ which has  a unique fixed point $z_n^M$ in $B (\varphi, \delta_0).$  Now we define another operator 
$\tilde{B}_n$  which also has $z_n^M$ as a fixed point.

From (\ref{eq:2.23})  recall that
\begin{equation*}
\tilde{\mathcal{K}}_n^M (x) = Q_n \mathcal {K}_m (x) + \mathcal {K}_m (Q_n x) - Q_n \mathcal {K}_m (Q_n x).
\end{equation*}
For $ m\geq m_2,$ define
\begin{eqnarray}\label{eq:5.7}
\tilde{B}_n (x) = \varphi_m - \left [ I - {\mathcal{K}}_m'   (\varphi_m) \right]^{-1}  \left \{  \mathcal {K}_m (\varphi_m) -  
\mathcal{K}_m'   (\varphi_m) \varphi_m - \tilde{\mathcal{K}}_n^M  ( x) +   \mathcal{K}_m'   (\varphi_m) x \right \}.
\end{eqnarray}
 Then
\begin{eqnarray*}
&& \tilde{B}_n (x) = x\\
&\Leftrightarrow& (I - {\mathcal{K}}_m'   (\varphi_m) ) \varphi_m -   \left \{  \mathcal {K}_m (\varphi_m) -  
\mathcal{K}_m'   (\varphi_m) \varphi_m - \tilde{\mathcal{K}}_n^M  ( x)  +   \mathcal{K}_m'   (\varphi_m) x\right \}\\
&& \hspace*{2 in} = (I - {\mathcal{K}}_m'   (\varphi_m) )  x\\
&\Leftrightarrow& x - \tilde{\mathcal{K}}_n^M  ( x) = \varphi_m  - \mathcal {K}_m (\varphi_m) = f.
\end{eqnarray*}
Thus, $x$ is a fixed point of $\tilde{B}_n$ if and only if it satisfies the equation (\ref{eq:2.24}).

Note that for $n \geq n_1$ and $ m \geq m_2,$ 
\begin{eqnarray*}
z_n^M - \varphi_m & = &  \tilde{B}_n (z_n^M) - \varphi_m \\
& = & - \left [ I - {\mathcal{K}}_m'   (\varphi_m) \right]^{-1}  \left \{  \mathcal {K}_m (\varphi_m) -  
\mathcal{K}_m'   (\varphi_m) \varphi_m - \tilde{\mathcal{K}}_n^M  ( z_n^M) +   \mathcal{K}_m'   (\varphi_m) z_n^M \right \}.
\end{eqnarray*}
Since ${\mathcal{K}}_m'   (\varphi_m)$ and $\left [ I - {\mathcal{K}}_m'   (\varphi_m) \right]^{-1}$ commute, the
first term in (\ref{eq:5.2}) can be written as
\begin{eqnarray*}
&&\mathcal{K}_m'   (\varphi_m) ( z_n^M - \varphi_m ) \\
& = &  
 - \left [ I - {\mathcal{K}}_m'   (\varphi_m) \right]^{-1}  \mathcal{K}_m'   (\varphi_m) \left \{  \mathcal {K}_m (\varphi_m) -  
\mathcal{K}_m'   (\varphi_m) \varphi_m - \tilde{\mathcal{K}}_n^M  ( z_n^M) +   \mathcal{K}_m'   (\varphi_m) z_n^M \right \}.
\end{eqnarray*}
We write
\begin{eqnarray}\nonumber
&&\mathcal{K}_m'   (\varphi_m) ( z_n^M - \varphi_m ) \\\nonumber
& = &  
 - \left [ I - {\mathcal{K}}_m'   (\varphi_m) \right]^{-1}  \mathcal{K}_m'   (\varphi_m) \left \{  \mathcal {K}_m (\varphi_m) -  \tilde{\mathcal{K}}_n^M (\varphi_m) \right \}\\\nonumber
&&+ \left [ I - {\mathcal{K}}_m'   (\varphi_m) \right]^{-1} \mathcal{K}_m'   (\varphi_m) \left \{
\tilde{\mathcal{K}}_n^M (z_n^M) - \tilde{\mathcal{K}}_n^M (\varphi_m)  - \left (  \tilde{\mathcal{K}}_n^M \right )' (\varphi_m) 
(z_n^M - \varphi_m) \right \}\\\label{eq:5.8}
&& + \left [ I - {\mathcal{K}}_m'   (\varphi_m) \right]^{-1} \mathcal{K}_m'   (\varphi_m) \left \{
\left (\left (  \tilde{\mathcal{K}}_n^M \right )' (\varphi_m)  -  \mathcal{K}_m'   (\varphi_m) \right  ) (z_n^M - \varphi_m) \right \}.
\end{eqnarray}
We now obtain error estimates for the quantities appearing in the above expression.
\end{remark}
\begin{proposition}\label{prop:5.3}
If $\displaystyle {\frac {\partial \kappa} {\partial u}  \in C^{3 r} (\Omega)}$ and $x \in C^{2 r} [a, b],$ then for $ n \geq n_1$ and $m \geq m_2,$
\begin{equation}\label{eq:5.9}
\| \mathcal {K}_m'  (\varphi_m) (I - Q_n) \mathcal {K}_m'  (\varphi_m) (I - Q_n) x \|_\infty = O (h^{4 r}).
\end{equation}
\end{proposition}
\begin{proof}
We have
\begin{equation}\nonumber
\mathcal {K}_m'  (\varphi_m) v (s)  =  \tilde {h}  \sum_{j=1}^m  \sum_{i=1}^\rho  w_i \;  \frac {\partial \kappa } {\partial u} (s,  \zeta_i^j, \varphi_m (\zeta_i^j)) v (\zeta_i^j), \;\;\; s \in [a, b].
\end{equation}
Let
\begin{equation*}
\ell_m (s, t) = \frac {\partial \kappa } {\partial u} (s, t, \varphi_m (t)), \;\;\; s, t \in [a, b].
\end{equation*}
Then $\ell_m \in C^{3 r} ([a, b] \times [a, b])$ and Proposition \ref{prop:3.2}  is applicable. Hence
\begin{eqnarray}\nonumber
\|\mathcal {K}_m'  (\varphi_m) (I - Q_n) x \|_\infty \leq C_4 \|\ell_m \|_{r, \infty}  \|x\|_{2 r, \infty} h^{2 r}.
\end{eqnarray}
Let 
$$ C_{9} = \max_{0 \leq i+ j \leq  r}  \max_{\stackrel {s, t \in [a, b]}{|u| \leq \|\varphi \|_\infty + \delta_0 }}
\left | \frac {\partial^{i + j  +1} \kappa } {\partial s^i \partial t^j\partial u} (s, t, u) \right |.$$
Since for $m \geq m_2,$ $\varphi_m \in B (\varphi, \delta_0),$ it follows that
\begin{equation}\nonumber
\|\ell_m \|_{r, \infty} \leq C_{9}.
\end{equation}
It then follows that
\begin{eqnarray}\label{eq:5.12}
\|\mathcal {K}_m'  (\varphi_m) (I - Q_n) x \|_\infty \leq C_4 C_{9}   \|x\|_{2 r, \infty} h^{2 r}.
\end{eqnarray}
Using the above estimate, we obtain
\begin{eqnarray}\label{eq:5.13}
\|\mathcal {K}_m'  (\varphi_m) (I - Q_n) \mathcal {K}_m'  (\varphi_m) (I - Q_n) x \|_\infty \leq C_4 C_{9} \|\mathcal {K}_m'  (\varphi_m) (I - Q_n) x\|_{2 r, \infty} h^{2 r}.
\end{eqnarray}
For $0 \leq \beta \leq 2 r,$
\begin{equation}\nonumber
\left (\mathcal {K}_m'  (\varphi_m) v \right )^{(\beta)} (s)  =  \tilde {h}  \sum_{j=1}^m  \sum_{i=1}^\rho  w_i \;  \frac {\partial^{\beta+1} \kappa } {\partial s^\beta\partial u} (s,  \zeta_i^j, \varphi_m (\zeta_i^j)) v (\zeta_i^j), \;\;\; s \in [a, b].
\end{equation}
Note that for $0 \leq \beta \leq 2 r,$
\begin{equation}\nonumber
\frac {\partial^{\beta+1} \kappa } {\partial s^\beta\partial u} (s, t, \varphi_m (t)) = D^{(\beta, 0)} \ell_m (s, t) \in
C^{3 r - \beta} ([a, b] \times [a, b]). 
\end{equation}
Let 
$$ C_{10} = \max_{ {0 \leq i + j \leq  3 r} } 
\max_{\stackrel {s, t \in [a, b]}{|u| \leq \|\varphi \|_\infty + \delta_0 }}
\left | \frac {\partial^{i  + j  +1} \kappa } {\partial s^{i } \partial t^j \partial u} (s, t, u) \right |.$$
Then for $0 \leq \beta \leq 2 r,$
\begin{equation}\nonumber
\|D^{(\beta, 0)} \ell_m \|_{r, \infty} \leq C_{10}.
\end{equation}
By Proposition \ref{prop:3.2},  for $0 \leq \beta \leq 2 r,$ we obtain
\begin{eqnarray*}
\|\left (\mathcal {K}_m'  (\varphi_m) (I - Q_n) x \right )^{(\beta)} \|_{ \infty}  & \leq & C_4 \|D^{(\beta, 0)} \ell_m \|_{r, \infty}  \|x\|_{2r, \infty} h^{2 r}\\
& \leq & C_4 C_{10} \|x\|_{2r, \infty} h^{2 r}.
\end{eqnarray*}
Hence
\begin{eqnarray}\nonumber
\|\mathcal {K}_m'  (\varphi_m) (I - Q_n) x\|_{ 2r, \infty}  & =&
\max \{ \|\left (\mathcal {K}_m'  (\varphi_m) (I - Q_n) x \right )^{(\beta)} \|_{ \infty}: 0 \leq \beta \leq 2 r\}\\\label{eq:5.14}
&\leq& C_4 C_{10} \|x\|_{2r, \infty} h^{2 r}.
\end{eqnarray}
Thus, from (\ref{eq:5.13}) and (\ref{eq:5.14}) we obtain
\begin{eqnarray}\nonumber
\|\mathcal {K}_m'  (\varphi_m) (I - Q_n) \mathcal {K}_m'  (\varphi_m) (I - Q_n) x \|_\infty\leq (C_4)^2 C_{9} C_{10}  \|x\|_{2r, \infty}  h^{4 r},
\end{eqnarray}
which completes the proof.
\end{proof}

\begin{proposition}\label{prop:5.4}
If $\displaystyle {\frac {\partial \kappa} {\partial u}  \in C^{3 r} (\Omega)},$  then for $ n \geq n_1$ and $m \geq m_2,$
\begin{equation}\label{eq:5.15}
\| \mathcal {K}_m'  (\varphi_m) (I - Q_n) \mathcal {K}_m'  (\varphi_m) \| = O (h^{2 r}).
\end{equation}
\end{proposition}
\begin{proof}
From (\ref{eq:5.12})
\begin{equation}\nonumber
\|\mathcal {K}_m'  (\varphi_m) (I - Q_n) \mathcal {K}_m'  (\varphi_m)x\|_{2 r, \infty}  \leq C_4 C_{9} \|\mathcal {K}_m'  (\varphi_m)x\|_{2r, \infty} h^{2 r}.
\end{equation}
 For  $0 \leq \beta \leq 2 r,$
\begin{eqnarray}\nonumber
\left (\mathcal {K}_m'  (\varphi_m) x \right )^{(\beta)} (s)  &=&  \tilde {h}  \sum_{j=1}^m  \sum_{i=1}^\rho  w_i \;  \frac {\partial^{\beta+1} \kappa } {\partial s^\beta\partial u} (s,  \zeta_i^j, \varphi_m (\zeta_i^j)) x (\zeta_i^j), \;\;\; s \in [a, b].
\end{eqnarray}
Then
\begin{eqnarray*}
\|\mathcal {K}_m'  (\varphi_m)    x \|_{2 r, \infty} \leq (b - a) \;  C_{10}\; \left ( \sum_{i=1}^\rho |w_i| \right ) \|x \|_\infty.
\end{eqnarray*}
Thus,
\begin{eqnarray}\nonumber
\| \mathcal {K}_m'  (\varphi_m)   (I - Q_n) \mathcal {K}_m'  (\varphi_m)   x \|_\infty 
& \leq & (b - a) \; C_4  C_{9} C_{10}  \left ( \sum_{i=1}^\rho |w_i| \right ) \|x \|_\infty \;  h^{2 r}.
\end{eqnarray}
Taking the supremum over the set $\{ x \in C [a, b]: \|x\|_\infty \leq 1 \},$ we obtain
\begin{eqnarray}\nonumber
\| \mathcal {K}_m'  (\varphi_m)   (I - Q_n) \mathcal {K}_m'  (\varphi_m)    \|
& \leq & (b - a) \; C_4 \; C_{9} C_{10}  \left ( \sum_{i=1}^\rho |w_i| \right ) \;  h^{2 r},
\end{eqnarray}
which completes the proof.
\end{proof}


\begin{proposition}\label{prop:5.5}
Let $ r \geq 1, $  $ \kappa \in C^d (\Omega)$  and $ f \in C^{d} [a, b].$ If $\displaystyle { \frac {\partial^2 \kappa} {\partial u^2}  \in C^{2 r} (\Omega)},$  then for $ n \geq n_1$ and $m \geq m_2,$
\begin{eqnarray}\nonumber
\left \| \mathcal {K}_m'  (\varphi_m)   (I - Q_n) \left [ \mathcal{K}_m (Q_n\varphi_m) - \mathcal{K}_m (\varphi_m) - \mathcal {K}_m'  (\varphi_m) (Q_n \varphi_m - \varphi_m ) \right ] \right \|_\infty = O (h^{4 r} ).\\\label{eq:5.17}
\end{eqnarray}
\end{proposition}
\begin{proof}
Let
\begin{eqnarray}\label{**}
y_n = \mathcal{K}_m (Q_n\varphi_m) - \mathcal{K}_m (\varphi_m) - \mathcal {K}_m'  (\varphi_m) (Q_n \varphi_m - \varphi_m ) 
= R (Q_n \varphi_m - \varphi_m ) .
\end{eqnarray}
By (\ref{eq:5.12})
\begin{eqnarray}\label{eq:5.18}
\| \mathcal{K}_m' (\varphi_m)(I - Q_n) y_n \|_\infty \leq C_4 C_{9} \; \|y_n \|_{2 r, \infty} \;  h^{2 r}.
\end{eqnarray}
Recall from (\ref{eq:4.3}) and (\ref{eq:4.4}) that
 \begin{eqnarray}\nonumber
R   (v_2 - v_1)  (s) 
&=& \int_0^1   (1 - \theta)   S_\theta  (v_2 - v_1 ) (s)   d \theta,
\end{eqnarray}
with 
\begin{eqnarray}\nonumber
S_\theta  (v_2 - v_1)  (s) 
& = & \tilde {h}  \sum_{j=1}^m  \sum_{i=1}^\rho  w_i \;  \frac {\partial^2 \kappa } {\partial u^2} \left (s,  \zeta_i^j, 
v_1  (\zeta_i^j) + \theta
 (v_2 - v_1)  (\zeta_i^j) \right ) (v_2 - v_1)^2 (\zeta_i^j).
\end{eqnarray}
Let
$$ C_{11} = \max_{0 \leq \beta \leq 2 r}  \max_{\stackrel {s, t \in [a, b]}{|u| \leq \|\varphi \|_\infty + \delta_0 }}
\left | \frac {\partial^{\beta +2} \kappa } {\partial s^\beta\partial u^2} (s, t, u) \right |.$$
Then for $0 \leq \beta \leq 2 r,$
\begin{equation}\nonumber
\|( R  (v_2 - v_1)  )^{(\beta)}\|_\infty \leq  \frac {C_{11} (b - a)} {2}  \left ( \sum_{i=1}^\rho | w_i | \right ) \|v_2 - v_1 \|_\infty^2
\end{equation}
and
\begin{eqnarray}\label{eq:5.19}
\|y_n \|_{2 r, \infty} = \|R (Q_n \varphi_m - \varphi_m )\|_{2 r, \infty } \leq 
\frac {C_{11} (b - a)} {2}  \left ( \sum_{i=1}^\rho | w_i | \right ) \|Q_n \varphi_m - \varphi_m \|_\infty^2.
\end{eqnarray}
From (\ref{eq:2.14}),  (\ref{eq:2.20}) and (\ref{eq:3.18}),   
\begin{eqnarray}\nonumber
\|(Q_n - I ) \varphi_m \|_\infty & \leq & \|(Q_n - I ) \varphi \|_\infty + \|(Q_n - I )( \varphi_m - \varphi ) \|_\infty \\\nonumber
& \leq & C_3 \|\varphi^{(r)}\|_\infty h^r + O \left (  \tilde{h}^{d} \right )
\\\label{eq:5.20}
&=& O (h^r).
\end{eqnarray}
The required result then follows from (\ref{eq:5.18}) - (\ref{eq:5.20}).
\end{proof}

\begin{proposition}\label{prop:5.6}
Let $ \kappa \in C^d (\Omega), $ $\displaystyle {\frac {\partial \kappa} {\partial u} \in C^{2 r} (\Omega)}$ and $ f \in C^{d} [a, b].$  
Then for $ n \geq n_1$ and $m \geq m_2,$
\begin{eqnarray}\nonumber
&&\left \|\tilde {\mathcal{K}}_n^M (z_n^M) - \tilde{\mathcal{K}}_n^M (\varphi_m)  - \left (  \tilde{\mathcal{K}}_n^M \right )' (\varphi_m) 
(z_n^M - \varphi_m) \right \|_\infty = O \left (\max \left \{ \tilde{h}^{d}, h^{3 r} \right \}^2 \right).
\\\label{eq:5.21}
\end{eqnarray}
\end{proposition}
\begin{proof}

Note that $\displaystyle {\varphi_m, z_n^M \in B \left (\varphi_m, \delta_0 \right ). }$
By the generalized Taylor's theorem,
\begin{eqnarray*}
&&\tilde {\mathcal{K}}_n^M (z_n^M) - \tilde{\mathcal{K}}_n^M (\varphi_m)  - \left (  \tilde{\mathcal{K}}_n^M \right )' (\varphi_m) 
(z_n^M - \varphi_m) \\
&& \hspace*{0.5 in} = 
\int_0^1 ( 1 - \theta) \left (  \tilde{\mathcal{K}}_n^M \right )'' \left (\varphi_m + \theta (z_n^M - \varphi_m) \right ) (z_n^M - \varphi_m)^2
(s) \; d \theta.
\end{eqnarray*}
Hence
\begin{eqnarray}\nonumber
&&\left \|\tilde {\mathcal{K}}_n^M (z_n^M) - \tilde{\mathcal{K}}_n^M (\varphi_m)  - \left (  \tilde{\mathcal{K}}_n^M \right )' (\varphi_m) 
(z_n^M - \varphi_m) \right \|_\infty \\\label{eq:5.22}
&& \hspace*{0.5 in} \leq \frac {1} {2} \max_{0 \leq \theta \leq 1}  
\left \|\left (\tilde {\mathcal{K}}_n^M \right )'' \left  (\varphi_m + \theta (z_n^M - \varphi_m) \right ) \right \| \| z_n^M - \varphi_m\|_\infty^2.
\end{eqnarray}
Note that for  $x \in B (\varphi, \delta_0),$ 
\begin{eqnarray}\nonumber
(\tilde{\mathcal{K}}_n^M)'' (x) = Q_n \mathcal{K}_m'' (x ) + (I - Q_n) \mathcal{K}_m'' (Q_n x ) (Q_n \otimes Q_n),
\end{eqnarray}
where $Q_n \otimes Q_n: X \times X \rightarrow X \times X$ is defined as
$$(Q_n \otimes Q_n) (v, w) = (Q_n v, Q_n w).$$
From (\ref{eq:2.10}) for $ s \in [a, b],$
\begin{eqnarray*}
&&(\mathcal {K}_m)''  \left  (\varphi_m + \theta (z_n^M - \varphi_m) \right ) (v_1, v_2)  (s) \\
&& \hspace*{0.5 in} =  \tilde {h}  \sum_{j=1}^m  \sum_{i=1}^\rho  w_i \;  \frac {\partial^2 \kappa } {\partial u^2} \left (s,  \zeta_i^j, 
\left  (\varphi_m + \theta  (z_n^M - \varphi_m) \right ) (\zeta_i^j)   \right  ) v_1 (\zeta_i^j) v_2 (\zeta_i^j).
\end{eqnarray*}
Then
 for $0 \leq \theta \leq 1,$
$$ \left \| (\mathcal {K}_m)''  \left  (\varphi_m + \theta (z_n^M - \varphi_m) \right ) \right \| \leq C_{6} (b - a) \left ( \sum_{i=1}^\rho |w_i| \right ),$$
where $C_6$ is defined in (\ref{eq:4.1}) with $\delta = \delta_0.$

In a similar manner, it can be shown that for $0 \leq \theta \leq 1,$ 
$$ \left \| (\mathcal {K}_m)''  \left  (Q_n \varphi_m + \theta (Q_n z_n^M - Q_n \varphi_m) \right ) \right \| \leq C_{6} (b - a) \left ( \sum_{i=1}^\rho |w_i| \right ).$$
Hence
\begin{eqnarray}\label{eq:5.23}
\max_{0 \leq \theta \leq 1}  
\left \|\left (\tilde {\mathcal{K}}_n^M \right )'' \left  (\varphi_m + \theta (z_n^M - \varphi_m) \right ) \right \|
\leq C_{6} (b - a) q (1 + q + q^2) \left ( \sum_{i=1}^\rho |w_i| \right ).
\end{eqnarray}
 By Theorem \ref{thm:4.6} and by (\ref{eq:2.14}),
$$ \|z_n^M - \varphi_m \|_\infty \leq \|z_n^M - \varphi \|_\infty + \|\varphi - \varphi_m \|_\infty= O \left (\max \left \{ \tilde{h}^{d},  h^{3 r} \right \} \right),$$
The required result follows from (\ref{eq:5.22}),  (\ref{eq:5.23}) and the above estimate..
\end{proof}

\begin{proposition}\label{prop:5.7}
Let $ \kappa \in C^d (\Omega), $ $\displaystyle {\frac {\partial \kappa} {\partial u} \in C^{3 r} (\Omega)}$ and $ f \in C^{d} [a, b].$  
Then for $ n \geq n_1$ and $m \geq m_2,$
\begin{eqnarray}\label{eq:5.24}
\left \|\mathcal{K}_m'   (\varphi_m) 
\left (\left (  \tilde{\mathcal{K}}_n^M \right )' (\varphi_m)  -  \mathcal{K}_m'   (\varphi_m) \right  ) (z_n^M - \varphi_m) 
\right \|_\infty
=O \left ( h^r \max \left \{ \tilde{h}^{d},  h^{3 r} \right \} \right).
\end{eqnarray}
\end{proposition}

\begin{proof}
Note that 
\begin{eqnarray}\nonumber
(\tilde{\mathcal{K}}_n^M)' (\varphi_m) = Q_n \mathcal{K}_m' ( \varphi_m ) + (I - Q_n) \mathcal{K}_m'  (Q_n \varphi_m ) Q_n.
\end{eqnarray}
Hence
\begin{eqnarray*}
\mathcal{K}_m'   (\varphi_m) \left (\left (  \tilde{\mathcal{K}}_n^M \right )' (\varphi_m)  -  \mathcal{K}_m'   (\varphi_m) \right  ) &= &
- \mathcal{K}_m'   (\varphi_m)  (I - Q_n) \mathcal{K}_m'   (\varphi_m) \\
& + & \mathcal{K}_m'   (\varphi_m) (I - Q_n) \mathcal{K}_m'  (Q_n \varphi_m ) Q_n\\
& = &\mathcal{K}_m'   (\varphi_m) (I - Q_n) (\mathcal{K}_m'  (Q_n \varphi_m )  -  \mathcal{K}_m'   (\varphi_m)) Q_n \\
& - & \mathcal{K}_m'   (\varphi_m)  (I - Q_n) \mathcal{K}_m'   (\varphi_m) (I - Q_n).
\end{eqnarray*}
By Proposition \ref{prop:5.4},
$$\| \mathcal{K}_m'   (\varphi_m)  (I - Q_n) \mathcal{K}_m'   (\varphi_m) \| = O (h^{ 2 r}).$$
By (\ref{eq:4.7}) and (\ref{eq:5.20}),
\begin{eqnarray*}
\|\mathcal{K}_m'  (Q_n \varphi_m )  -  \mathcal{K}_m'   (\varphi_m)\|\leq \gamma \|Q_n \varphi_m - \varphi_m \|_\infty = O (h^r).
\end{eqnarray*}
Let
$$ C_{12} =   \max_{\stackrel {s, t \in [a, b]}{|u| \leq \|\varphi \|_\infty + \delta_0 }}
\left | \frac {\partial \kappa } {\partial u} (s, t, u) \right |.$$
Then
\begin{eqnarray*}
\|\mathcal{K}_m'   (\varphi_m)  x \|_\infty \leq C_{12} (b - a) \left ( \sum_{i=1}^\rho |w_i| \right ) \|x\|_\infty.
\end{eqnarray*}
Hence
$$\|\mathcal{K}_m'   (\varphi_m) \| \leq C_{12} (b - a) \left ( \sum_{i=1}^\rho |w_i| \right ).$$
Thus,
\begin{eqnarray*}
\left \|\mathcal{K}_m'   (\varphi_m) \left (\left (  \tilde{\mathcal{K}}_n^M \right )' (\varphi_m)  -  \mathcal{K}_m'   (\varphi_m) \right  ) \right \|= O (h^r).
\end{eqnarray*}
Since
$$ \|z_n^M - \varphi \|_\infty = O \left (\max \left \{ \tilde{h}^{d},  h^{3 r} \right \} \right),$$
the required result follows.
\end{proof}

In the following theorem we obtain the order of convergence in the discrete iterated modified projection method.


\begin{theorem}\label{thm:5.8}
Let $ r \geq 1, $ $ \kappa \in C^d (\Omega), $ $\displaystyle {\frac {\partial^2 \kappa} {\partial u^2} \in C^{3 r} (\Omega)}$ and $ f \in C^{d} [a, b].$ 
Let $\varphi$ be the unique solution of (\ref{eq:1.2}) and assume that $1$ is not an eigenvalue of $\mathcal{K}' (\varphi).$ 
Let $\mathcal{X}_n$ be the space of piecewise polynomials of degree $\leq r - 1$ with respect to the partition (\ref{eq:2.15})
and $Q_n: L^\infty [0, 1] \rightarrow \mathcal {X}_n$ be the interpolatory projection at $r$ Gauss points defined by
(\ref{eq:2.19}). Let $z_n^M $ be the unique solution of (\ref{eq:2.24}) in $B (\varphi, \delta_0).$ Let $\tilde{z}_n^M$ be the discrete iterated modified projection solution defined by (\ref{eq:5.1}). Then
\begin{eqnarray}\label{eq:5.25}
\|\tilde{z}_n^M - \varphi \|_\infty = O (h^r \max \{\tilde{h}^{d}, h^{3 r} \}).
\end{eqnarray}
\end{theorem}
\begin{proof}
Recall from (\ref{eq:5.2}) that
\begin{equation}\label{eq:5.26}
\tilde{z}_n^M - \varphi_m = \mathcal {K}_m' (\varphi_m) (z_n^M - \varphi_m) + O (\max \{\tilde{h}^{d}, h^{3 r} \}^2)
\end{equation}
and from (\ref{eq:5.8}) that
\begin{eqnarray}\nonumber
&&\mathcal{K}_m'   (\varphi_m) ( z_n^M - \varphi_m ) \\\nonumber
& = &  
 - \left [ I - {\mathcal{K}}_m'   (\varphi_m) \right]^{-1}  \mathcal{K}_m'   (\varphi_m) \left \{  \mathcal {K}_m (\varphi_m) -  \tilde{\mathcal{K}}_n^M (\varphi_m) \right \}\\\nonumber
&&+ \left [ I - {\mathcal{K}}_m'   (\varphi_m) \right]^{-1} \mathcal{K}_m'   (\varphi_m) \left \{
\tilde{\mathcal{K}}_n^M (z_n^M) - \tilde{\mathcal{K}}_n^M (\varphi_m)  - \left (  \tilde{\mathcal{K}}_n^M \right )' (\varphi_m) 
(z_n^M - \varphi_m) \right \}\\\label{eq:5.27}
&& + \left [ I - {\mathcal{K}}_m'   (\varphi_m) \right]^{-1} \mathcal{K}_m'   (\varphi_m) \left \{
\left (\left (  \tilde{\mathcal{K}}_n^M \right )' (\varphi_m)  -  \mathcal{K}_m'   (\varphi_m) \right  ) (z_n^M - \varphi_m) \right \}.
\end{eqnarray}
From (\ref{eq:5.6}) we have
$$ \left \|\left [ I - {\mathcal{K}}_m'   (\varphi_m) \right]^{-1} \right \| \leq 2 C_{7}.$$
Note that
\begin{eqnarray}\nonumber
\mathcal{K}_m (\varphi_m) - \tilde{\mathcal{K}}_n^M (\varphi_m)  &= & (I - Q_n) (\mathcal{K}_m (\varphi_m) - \mathcal{K}_m (Q_n\varphi_m))\\\nonumber
& = & - (I - Q_n) (\mathcal{K}_m (Q_n\varphi_m) - \mathcal{K}_m (\varphi_m) - \mathcal {K}_m'  (\varphi_m) (Q_n \varphi_m - \varphi_m ) )\\\nonumber
&& -  (I - Q_n) \mathcal {K}_m'  (\varphi_m) (Q_n \varphi_m - \varphi_m ). 
\end{eqnarray}
By Proposition \ref{prop:5.5},
\begin{eqnarray}\nonumber
\left \|\mathcal{K}_m'   (\varphi_m) (I - Q_n) (\mathcal{K}_m (Q_n\varphi_m) - \mathcal{K}_m (\varphi_m) - \mathcal {K}_m'  (\varphi_m) (Q_n \varphi_m - \varphi_m ) \right \|_\infty = O (h^{4 r}).\\\label{eq:5.28}
\end{eqnarray}
Since
\begin{eqnarray*}\nonumber
&&\left \|\mathcal {K}_m'  (\varphi_m)   (I - Q_n) \mathcal {K}_m'  (\varphi_m)   (I - Q_n) \varphi_m \right \|_{\infty}
\\
 & & \hspace*{1 in} \leq
\left \|\mathcal {K}_m'  (\varphi_m)   (I - Q_n) \mathcal {K}_m'  (\varphi_m)   (I - Q_n) \varphi \right \|_{ \infty} \\
&  & \hspace*{1 in} +  (1 + q)  \; \left \|\mathcal {K}_m'  (\varphi_m)   (I - Q_n) \mathcal {K}_m'  (\varphi_m)  \right \| \| \varphi_m - \varphi  \|_{ \infty}, 
\end{eqnarray*}
by Proposition \ref{prop:5.3}, Proposition \ref{prop:5.4} and the estimate (\ref{eq:2.14}), we obtain
\begin{equation}\label{eq:5.29}
\left \|\mathcal {K}_m'  (\varphi_m)   (I - Q_n) \mathcal {K}_m'  (\varphi_m)   (I - Q_n) \varphi_m \right \|_{\infty} = O (h^{4 r}).
\end{equation}
It then follows that
\begin{eqnarray}\label{eq:5.30}
\left \|\left [ I - {\mathcal{K}}_m'   (\varphi_m) \right]^{-1}  \mathcal{K}_m'   (\varphi_m) \left ( \mathcal{K}_m (\varphi_m) - \tilde{\mathcal{K}}_n^M (\varphi_m) \right )
\right \|_\infty =  O (h^{4 r}).
\end{eqnarray}
By Proposition \ref{prop:5.6},
\begin{eqnarray}\nonumber
&&\left \|\tilde {\mathcal{K}}_n^M (z_n^M) - \tilde{\mathcal{K}}_n^M (\varphi_m)  - \left (  \tilde{\mathcal{K}}_n^M \right )' (\varphi_m) 
(z_n^M - \varphi_m) \right \|_\infty = O \left (\max \left \{ \tilde{h}^{d}, h^{3 r} \right \}^2 \right).
\\\label{eq:5.31}
\end{eqnarray}
whereas by Proposition \ref{prop:5.7},
\begin{eqnarray}\label{eq:5.32}
\left \|\mathcal{K}_m'   (\varphi_m) \left \{
\left (\left (  \tilde{\mathcal{K}}_n^M \right )' (\varphi_m)  -  \mathcal{K}_m'   (\varphi_m) \right  ) (z_n^M - \varphi_m) \right \} \right \|_\infty
= O \left (h^r \max \left \{ \tilde{h}^{d},  h^{3 r} \right \} \right).
\end{eqnarray}
The required result follows from (\ref{eq:5.26})-(\ref{eq:5.32}).
\end{proof}
\setcounter{equation}{0}
\section{Numerical Results}
For the sake of illustration, we quote the following results from Grammont et al \cite{Gram3}.

Consider
\begin{equation}\label{eq:6.1}
 \varphi (s) - \int_0^1 \frac {d s} {s + t + \varphi (t)}   = f (s), \;\;\; 0 \leq s \leq 1,
 \end{equation}
where $f$ is so chosen that
$$ \varphi (t) = \frac {1} { t + c}, \;\;\; c > 0,$$
is a solution of (\ref{eq:6.1}).

We consider $X_n$ to be either  the space of piecewise constant functions or piecewise linear functions with respect
to the following uniform partition of $[0, 1]:$
\begin{equation}\label{eq:6.2}
0 < \frac{1}{n} <\frac{2}{n} < \cdots < \frac{n}{n}=1.
\end{equation}
The projection $Q_n$ is chosen to be the  interpolatory projection at $r $ Gauss points  with $r = 1$ or $ r = 2.$ 
Hence by Theorem  \ref {thm:4.6} and Theorem \ref{thm:5.8}
\begin{eqnarray}\label{eq:6.3}
\|z_n^M - \varphi \|_\infty = O (\max \{\tilde{h}^{d}, h^{3 r} \}), \;\;\; \|\tilde{z}_n^M - \varphi \|_\infty = O (h^r \max \{\tilde{h}^{d}, h^{3 r} \}).
\end{eqnarray}
If $X_n$ is the space of piecewise 
constant functions with respect to the partition (\ref{eq:5.2}), then 
we choose the composite Simpson rule with respect to the partition (\ref{eq:6.2}) to evaluate
the integrals numerically. Then $\tilde {h} = h$ and $d = 4.$ Thus,
\begin{eqnarray}\label{eq:6.4}
\|z_n^M - \varphi \|_\infty = O ( h^{3 } ), \;\;\; \|\tilde{z}_n^M - \varphi \|_\infty =  O ( h^{4} ).
\end{eqnarray}
If $X_n$ is the space of piecewise linear 
functions with respect to the partition (\ref{eq:5.2}), and the interpolation points 
are Gauss 2 points, then we choose the composite Gauss 2 point rule with respect to an uniform
partition with $m = n^2$ subintervals as the approximate quadrature rule. 
Then, $\tilde{h} = h^2$ and $ d = 4.$
Hence
\begin{eqnarray}\label{eq:6.5}
\|z_n^M - \varphi \|_\infty = O ( h^{6 } ), \;\;\; \|\tilde{z}_n^M - \varphi \|_\infty =  O ( h^{8} ).
\end{eqnarray}
In the following tables $\delta_m$ and $\delta_{IM}$ denote the computed orders of convergence in the 
discrete modified projection method and the discrete iterated modified projection method, respectively. 

\begin{center}
Table $6.1$ 
$$\varphi (t) = \frac {1} {t + 1}$$
\begin{tabular} {|c|cc|cc|cc|cc|}\hline
& \multicolumn{4} {|c|} {Piecewise Constant: $r = 1$} & \multicolumn{4} {|c|} {Piecewise Linear: $r = 2$}\\\hline
$n$ & $\| \varphi - z_n^M \|_\infty$ & $\delta_M$ & $\| \varphi  - \tilde{z}_n^M\|_\infty$ & $\delta_{IM}$
& $\| \varphi - z_n^M \|_\infty$ & $\delta_M$ & $\| \varphi  - \tilde{z}_n^M\|_\infty$ & $\delta_{IM}$ \\
\hline
2& $ 8.46 \times 10^{-4}   $ &               & $ 2.38 \times 10^{-5}   $   &    &     $ 5.06 \times 10^{-4}   $            &  & $ 6.47 \times 10^{-5}   $              & \\
4&  $  1.03 \times 10^{-4} $ & $ 3.04 $ & $  1.37 \times 10^{-6} $ & $ 4.12$  &   $ 1.07 \times 10^{-5}   $ & $ 5.56 $ & $ 2.09 \times 10^{-7}   $ & $ 8.27 $\\
8&  $  1.24 \times 10^{-5} $ & $ 3.05 $ & $ 8.18 \times 10^{- 8}  $ &  $ 4.07$ &   $ 1.85 \times 10^{- 7}  $ & $ 5.86 $ & $ 8.45 \times 10^{- 10} $ & $ 7.95 $\\
16& $ 1.45 \times 10^{-6} $ & $ 3.09 $ & $ 4.99 \times 10^{-9}  $ & $ 4.04$  &   $ 3.07 \times 10^{- 9}  $ & $ 5.90 $ & $ 3.35 \times 10^{- 12} $ & $ 7.98 $\\
32& $ 1.59 \times 10^{-7}  $ & $ 3.19 $ & $ 3.08 \times 10^{-10}  $ &  $ 4.02$ &   $ 4.74 \times 10^{- 11}   $ & $ 6.02$ & $ 1.34 \times 10^{-14}$ & $ 7.96 $ \\\hline
\end{tabular}

\end{center}
 \newpage
\begin{center}
Table $6.2:$ $\hspace*{1 cm} \varphi (t) = \frac {1} {t + 0.1}$
\vspace*{0.2 cm}

\begin{tabular} {|c|cc|cc|cc|cc|}\hline
& \multicolumn{4} {|c|} {Piecewise Constant: $r = 1$} & \multicolumn{4} {|c|} {Piecewise Linear: $r = 2$}\\\hline
$n$ & $\| \varphi - z_n^M \|_\infty$ & $\delta_M$ & $\| \varphi  - \tilde{z}_n^M\|_\infty$ & $\delta_{IM}$
& $\| \varphi - z_n^M \|_\infty$ & $\delta_M$ & $\| \varphi  - \tilde{z}_n^M\|_\infty$ & $\delta_{IM}$ \\
\hline
2& $ 3.64 \times 10^{-4}   $ &               & $ 7.80 \times 10^{-6}   $   &    &     $ 9.39 \times 10^{-5}   $            &  & $ 1.14 \times 10^{-4}   $              & \\
4&  $  6.29 \times 10^{-5} $ & $ 2.53 $ & $  4.20 \times 10^{-7} $ & $ 4.21$  &   $ 1.19 \times 10^{-4}   $ & $ -0.35 $ & $ 2.84 \times 10^{-7}   $ & $ 8.65 $\\
8&  $  6.85 \times 10^{-6} $ & $ 2.83 $ & $ 2.42 \times 10^{- 8}  $ &  $ 4.12$ &   $ 3.30 \times 10^{- 6}  $ & $ 5.18 $ & $ 1.10 \times 10^{- 9} $ & $ 8.01 $\\
16& $ 1.12 \times 10^{-6} $ & $ 2.99 $ & $ 1.45 \times 10^{-9}  $ & $ 4.06 $  &   $ 4.99 \times 10^{- 8}  $ & $ 6.05 $ & $ 4.35 \times 10^{- 12} $ & $ 7.99 $\\
32& $ 1.27 \times 10^{-7}  $ & $ 3.14 $ & $ 8.89 \times 10^{-11}  $ &  $ 4.03 $ &   $ 7.00 \times 10^{- 10}   $ & $ 6.16$ & $ 1.78 \times 10^{-14}$ & $ 7.93 $ \\\hline
\end{tabular}

\end{center}

\vspace*{0.2 cm}

It is seen that the computed orders of convergence match well with the expected orders of convergence in (\ref{eq:6.4}) 
and (\ref{eq:6.5}).
\section{Conclusion}
In this paper we consider approximate solutions of Urysohn integral equations with smooth kernels using discrete versions of the Modified Projection Method and of the Iterated Modified Projection Method
associated with an interpolatory projection at $r$ Gauss points. The interval $[a, b]$ is divided into $m$ subintervals each of length 
$\displaystyle {\tilde{h} = \frac {b - a} {m}}$ 
and a composite numerical quadrature with a degree of precision $d$ is chosen to replace all the integrals.  The range of the interpolatory projection at $r$ Gauss points is a space of piecewise polynomials of degree $\leq r - 1$ with respect 
to a uniform partition of $[a, b]$ with $n$ subintervals each of length $\displaystyle {h = \frac {b - a} {n}.}$  The exact solution is denoted by $\varphi$ and the approximate solutions obtained by using the Discrete Modified Projection Method and the Discrete Iterated Modified Projection Method  are denoted respectively by $z_n^M$ and $\tilde{z}_n^M. $ We choose $ m = p \; n$ where $ p \in \N.$ The following orders of convergence are proved:
\begin{eqnarray*}
\|z_n^M - \varphi \|_\infty = O (\max \{\tilde{h}^{d}, h^{3 r} \}), \;\;\; 
\|\tilde{z}_n^M - \varphi \|_\infty = O (h^r \max \{\tilde{h}^{d}, h^{3 r} \}).
\end{eqnarray*}
Since the errors in the  Modified Projection Method and the Iterated Modified Projection Method are respectively of the order of $h^{3 r}$ and $h^{4 r},$
these orders of convergence are preserved if the numerical quadrature formula is so chosen that $\tilde{h}^{d} = h^{3 r}.$ Note that we have at our disposal $\tilde {h},$ that is $m,$ and $d$ to achieve this equality.

\noindent
{\bf Acknowledgement:}
The  first  author would like to thank  Indo-French Centre for Applied Mathematics 
(IFCAM)  for the partial support.

\end{document}